\documentclass[12pt]{article}

\usepackage[leqno]{amsmath}
\usepackage{amssymb,amsfonts,xfrac,MnSymbol}
\usepackage{amsthm}
\usepackage{cite}
\usepackage{hyperref}
\usepackage{todonotes}
\usepackage{enumitem}
\usepackage{graphicx}
\usepackage{tikz-cd, tikz}
\usepackage{color}
\usepackage{import}
\usepackage{cleveref}
\usepackage{soul}
\usepackage{setspace}
\usepackage[margin=1.2in]{geometry}
\usepackage{comment}

\numberwithin{equation}{section}

\usepackage{thm-restate}

\newtheorem{theorem}{Theorem}[section]
\newtheorem{claim}[theorem]{Claim}
\newtheorem{proposition}[theorem]{Proposition}
\newtheorem{lemma}[theorem]{Lemma}
\newtheorem{corollary}[theorem]{Corollary}

\newtheorem*{theorem*}{Theorem}
\newtheorem*{claim*}{Claim}
\newtheorem*{proposition*}{Proposition}
\newtheorem*{lemma*}{Lemma}
\newtheorem*{corollary*}{Corollary}

\theoremstyle{definition}
\newtheorem{definition}[theorem]{Definition}
\newtheorem{observation}[theorem]{Observation}

\newtheorem{remark}[theorem]{Remark}

\newtheorem{question}[theorem]{Question}

\newtheorem{fact}[theorem]{Fact}
\newtheorem{notation}[theorem]{Notation}

\newtheorem*{definition*}{Definition}
\newtheorem*{observation*}{Observation}
\newtheorem*{remark*}{Remark}
\newtheorem*{example*}{Example}
\newtheorem*{question*}{Question}
\newtheorem*{exercise*}{Exercise}
\newtheorem*{fact*}{Fact}
\newtheorem*{notation*}{Notation}

\newcommand{\bbG}{\mathbb{G}}

\newcommand{\bbN}{\mathbb{N}}

\newcommand{\bbU}{\mathbb{U}}

\newcommand{\bfD}{\mathbf{D}}

\newcommand{\bfd}{\mathbf{d}}

\newcommand{\calA}{\mathcal{A}}
\newcommand{\calB}{\mathcal{B}}
\newcommand{\calC}{\mathcal{C}}
\newcommand{\calD}{\mathcal{D}}
\newcommand{\calE}{\mathcal{E}}

\newcommand{\calH}{\mathcal{H}}

\newcommand{\calN}{\mathcal{N}}

\newcommand{\calQ}{\mathcal{Q}}
\newcommand{\calR}{\mathcal{R}}
\newcommand{\calS}{\mathcal{S}}

\newcommand{\calX}{\mathcal{X}}
\newcommand{\calY}{\mathcal{Y}}

\newcommand{\cl}[1]{\overline{#1}}
\newcommand{\actson}{\curvearrowright}

\newcommand{\inj}{\hookrightarrow}

\newcommand{\normalin}{\lhd}

\newcommand{\ii}{^{-1}}

\newcommand{\gen}[1]{\left< #1 \right>}

\newcommand{\tild}[1]{\widetilde{#1}}

\newcommand{\nea}{\nearrow}
\newcommand{\sea}{\searrow}
\newcommand{\nwa}{\nwarrow}

\DeclareMathOperator{\Stab}{Stab}

\DeclareMathOperator{\id}{id}
\DeclareMathOperator{\Aut}{Aut}
\DeclareMathOperator{\Sym}{Sym}
\DeclareMathOperator{\Alt}{Alt}
\DeclareMathOperator{\Fix}{Fix}

\DeclareMathOperator{\Comm}{Comm}

\DeclareMathOperator{\Lk}{Lk}
\DeclareMathOperator{\supp}{Supp}

\DeclareMathOperator{\pr}{pr}
\DeclareMathOperator{\str}{str}
\DeclareMathOperator{\im}{Im}

\title{Simple Lattices in Products of Davis Complexes}
\author{Michal Amir and Nir Lazarovich}
\date{}

\begin{document}
\maketitle

\begin{abstract}
Burger and Mozes \cite{burger1997finitely} constructed the first examples of simple uniform lattices in products of trees. 
In this paper, we construct simple uniform lattices in products of certain Davis complexes. 
More precisely, we consider lattices in products of trees and two-dimensional Davis complexes of the right-angled Coxeter group whose defining graph is an odd graph.
As part of the proof, we define an analogue of the Burger-Mozes universal groups in this setting, and provide a local criterion for a vertex transitive group to be dense in the universal group.





\end{abstract}

\section{Introduction}

\subsection{Background}
Simple groups are often seen as the building blocks of group theory. While finite simple groups were completely classified, finding infinite simple groups is still an active field of research. 
Among them, the search for finitely presented infinite simple groups has been of particular interest. There are only a few known families of such groups.  
The first example of a finitely presented simple group 
was given by Thompson 
in 1965. 
Thompson's group $T$ is a group of homeomorphisms of the circle.
Generalizations of this example, and of the simplicity of certain groups of homeomorphisms, were 
later developed by Higman \cite{higman1974finitely}, 
Scott \cite{scott1984construction},
Brown \cite{brown1987finiteness},
Stein \cite{stein1992groups}, R{\"o}ver \cite{rover1999constructing},  
and more recently, Lodha \cite{lodha2019finitely}. 
We refer the reader to the references, and references within for more information. 

A different family of examples arises from  lattices in products. Recall that a subgroup $\Gamma$ of a locally compact topological group $G$ is a \emph{lattice} if it is discrete and $G/\Gamma$ supports a finite volume $G$ invariant measure. Such a lattice is \emph{uniform} if moreover $G/\Gamma$ is compact. We will say that a group is a lattice in a space $X$ if it is a lattice in the group of isometries of $X$. Since the spaces we will consider are complexes, we will denote this group by $\Aut(X)$.
Burger and Mozes \cite{
burger2000groups}  first constructed  uniform lattices acting on two regular trees. Independently, Wise \cite{wise2007complete}  
found an example of non-residually-finite uniform lattice in products of trees. 
These groups have finite index in groups that act simply transitively on the vertex set of a product of two regular trees. Such groups are known today as \emph{BMW} groups (after Burger, Mozes and Wise). More examples of virtually simple BMW groups were found by Rattaggi  \cite{rattaggi2004computations} and 
Radu \cite{radu2020new}. Moreover, Lazarovich-Levcovitz-Margolis \cite{lazarovich2022counting} 
showed that simplicity is ubiquitous among BMW groups. 
Titz-Mite and Witzel \cite{mite2023c2,mite2025non} produced non-residually-finite uniform lattices in $\tild C_2$-buildings, these are expected to be virtually simple. 

In higher dimensions, Caprace and R\'emy \cite{caprace2009simplicity} proved that certain Kac-Moody groups are simple. These groups are  non-uniform lattices in products of two buildings. These buildings are not necessarily trees, and so their examples include spaces of dimension higher than two. 
Recently, Hughes \cite{hughes2022lattices} constructed non-residually-finite uniform lattices in higher dimensional CAT(0) cube complexes.

\subsection{Main Results}
The goal of this paper is to construct the first examples  of uniform simple lattices acting on  three-dimensional spaces.
These spaces will be products of two Davis complexes -- a regular tree and the right-angled Davis complex of an Odd graph. 

Let $d\in \bbN$, the Odd graph $O_d$ is the following $d$-regular graph: The vertex set of $O_d$ is the set $\binom{[2d-1]}{d-1}$ of subsets of $\{1,\dots,2d-1\}$ of size $d-1$. Two sets are connected by an edge if they are disjoint.
Let $X=X_{O_d}$ be the right-angled Davis complex associated to the graph $O_d$.
That is, $X$ is obtained from the Cayley graph of the right angled Coxeter group $W_L$ by collapsing bigons to edges, and filling in 
cubes whenever possible (cf. \S\ref{the davis complex}). 
Denote by $T_c$ the $c$-regular tree.

The main result of this paper is the following: 

\begin{restatable}{theoremA}{simpleLatticesInProducts}
\label{simple Lattices In Products}
    For infinitely many $c,d$ 
    there exists a finitely-presented, simple, uniform lattice in $\Aut( T_c )\times \Aut(X_{O_d})$. 
\end{restatable}

The proof of this theorem follows a similar strategy to that of Burger-Mozes \cite{burger2000groups,burger2000lattices}.
The general strategy for proving such a theorem is finding a non-residually-finite uniform lattice $\Lambda' \le \Aut( T_c )\times \Aut(X_{O_d})$ whose projections 
on each factor are dense in some simple topological groups. 
In our case, these simple topological groups are an adaptation of the universal groups of Burger-Mozes (cf. \Cref{def: universal group}).
Besides being simple groups, the remarkable property of the universal groups is that in order to check that a group (in our case, the projection of $\Lambda'$) is dense in them it suffices to check it locally.
Namely, to check that the restriction of a vertex stabilizer to the ball of radius 2 around this vertex matches that of the universal group.
Burger and Mozes proved this \emph{local-to-global} property for the universal group of a regular tree, the following is the analogous theorem for the Davis complex $X_{O_d}$ (cf. \Cref{def: universal group} for definitions):


\begin{restatable}{theoremA}{localToGlobal}
\label{local to global}\label{characterization of dense subgroups}
Let $6\le d\in \bbN$. Let $G$ be a vertex transitive subgroup of the universal group $U=U(\Alt_{2d-1})\le \Aut(X_{O_d})$. Then, the following are equivalent:
\begin{enumerate}
    \item $G$ is dense in $U$.
    \item $G$ is non-discrete and $G_v|_{\calB_1(v)} = U_v|_{\calB_1(v)}$ for some vertex $v$ of $X_{O_d}$.
    \item $G_v|_{\calB_2(v)} = U_v|_{\calB_2(v)}$ for some vertex $v$ of $X_{O_d}$.
\end{enumerate}
\end{restatable}

Assuming that such a uniform lattice $\Lambda'$ exists, the proof proceeds as follows: Since $\Lambda'$ is non-residually finite, its finite residue  $\Lambda=\bigcap \{H\le \Lambda' \;\mid\; |\Lambda':H|<\infty\}$  is a non-trivial normal subgroup, without any finite index subgroups.
By Bader-Shalom \cite{bader2006factor},  $\Lambda$ has finite index in $\Lambda'$ and 
is just-infinite, i.e., has no non-trivial infinite-index normal subgroups.  
We conclude that $\Lambda$ is simple. Note that since $\Lambda$ has finite index in the uniform lattice $\Lambda'$, $\Lambda$ is itself a uniform lattice, and as such it must be  finitely-presented.

We conclude this introduction with an outline of the paper: In \Cref{preliminaries} we cover the basics of CAT(0) cube complexes, normal paths and our notation, in \Cref{the davis complex} we define and discuss the Davis complex $X_{O_d}$, 
and define and prove relevant properties of the universal group of $X_{O_d}$, in \Cref{local to global section} we prove \Cref{local to global}. In \Cref{construction} we show how one can produce such lattices, and finally, we prove \Cref{simple Lattices In Products} in \Cref{proof section}. 


\section{CAT(0) cube complexes.}\label{preliminaries}

\subsection{CAT(0) cube complexes and hyperplanes. }

We provide a brief introduction to the basic definitions of CAT(0) cube complexes and their hyperplanes.
For more details and examples, we refer the reader to \cite{bridson2013metric,bestvina2014geometric}. 

\begin{definition}
     For a (CW) complex $X$, the \textit{$k$-th skeleton} of
     $X$, denoted by $X^k$ is the union of the cells of $X$ of dimension $\leq k$. 
\end{definition}

\begin{definition}
   A simplicial complex $\Sigma$ is \textit{flag} if every clique is the 1-skeleton of a simplex in $\Sigma$
\end{definition}

\begin{definition}
    For a complex $X$, the \textit{link} of a vertex $v\in X^{0}$, 
    $\Lk(v,X)$,  is the simplicial complex corresponding to a small sphere around $v$, that is, for every cube of dimension $i\geq 1$ incident to $v$, there is a unique simplex of dimension $i-1$ in $\Lk(v,X)$.
\end{definition}

\begin{definition}
    A \emph{cube complex} is a complex made up by gluing a collection of unit cubes $[0,1]^d$ (of varying dimensions $d$) along isometric faces.  

    A finite dimensional cube complex is \textit{CAT(0)} if it is connected, simply connected and the link of each of its vertices is a flag simplicial complex.
\end{definition}

\begin{remark}
    We note that CAT(0) is a metric property of non-positive curvature introduced by Gromov. He showed \cite{gromov1987hyperbolic} that when endowing each cube with the Euclidean metric, the obtained metric is CAT(0) if and only if the cube complex is CAT(0) in the sense of the previous definition. 
\end{remark}


\begin{definition}[Sageev \cite{sageev1995ends}]
    A \emph{midcube} of a cube $C=[0,1]^d$ is a subset of the form $\{\tfrac12\} \times [0,1]^{d-1}$.
    
    A \textit{hyperplane} in a CAT(0) space $X$, is the union of (non-empty) minimal collection of midcubes of the cubes of $X$ which is closed under inclusion in the following sense: if $M\subseteq M'$ are midcubes and one of them is in the collection then so is the other. 
    
    Each hyperplane is itself a CAT(0) cube complex, and separates $X$ into two components which are called \emph{halfspaces}.
\end{definition}

\subsection{The $\ell_\infty$ metric and normal paths}
CAT(0) cube complexes can be considered with various metrics. 
When endowing each cube with the $\ell_\infty$ metric, one obtains the so-called \emph{$\ell_\infty$ metric} on the cube complex. 
While geodesics in this metric are not unique, the following definition of ``normal paths'' due to Niblo-Reeves \cite[Definition 3.1]{niblo1998geometry} gives rise to a particular unique choice of $\ell_\infty$ geodesics (i.e., a geodesic bicombing).

\begin{definition}\label{normal cube path}
    A \emph{cube path} is a sequence $\{C_i\}_{i=1}^n$ of cubes of dimension at least one 
    with a corresponding vertex sequence $\{v_i\}_{i=0}^n$, such that $C_{i-1}\bigcap C_i=v_{i-1}$ for $2\leq i\leq n$, $v_0\in C_1$ is in the opposite end of $v_{1}$ and $v_n\in C_n$ is in the opposite end of $v_{n-1}$. 
    
    We say it is a \emph{normal cube path} if in addition 
 $\str(C_{i-1})\cap C_i=v_{i-1}$  
where $\str(C)$ is the  union of all closed cubes containing $C$.
\end{definition}

\begin{definition}\label{normal vertex path}
A \textit{normal (vertex) path}, $[v\nea x]$, from $v$ to $x$ is the sequence of vertices $\{v=v_0,v_1,\ldots, v_k=x\}$ corresponding to a normal cube path from $v$ to $x$. See \Cref{fig: normal path} (left).
 \end{definition} 

    \begin{figure}
        \centering
        \includegraphics[]{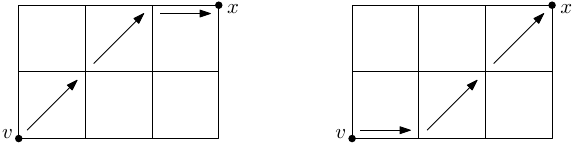}
        \caption{The normal paths $[v\nea x]$ (left) and $[v \sea x]$ (right)}\label{fig: normal path}
    \end{figure}

\begin{proposition}[{\cite[Proposition 3.3]{niblo1998geometry}}]
     For any two vertices $v,x\in X^0$, there is a unique  normal  path $[v\nea x]$ from $v$ to $x$.
\end{proposition}

Intuitively, the normal path $[v\nea x]$ is the path obtained by starting at $v$ and at each step go across the longest diagonal of a cube that decreases the distance to $x$.

\begin{definition}
    If in \Cref{normal cube path,normal vertex path} the cube path $\{C_i\}_{i=1}^k$ satisfies $C_{i-1}\cap \str(C_i)=v_{i-1}$, we denote the corresponding vertex path $\{v_i\}_{i=0}^k$ by $[v\sea x]$. See \Cref{fig: normal path} (right).
\end{definition}

The path $[v\sea x]$ is nothing more than the reverse path of $[x\nea v]$.
We remark that in general the normal paths  $[v\nea x]$ and $[v\sea x]$ are not equal as sets of vertices, but their length is the same, since both normal paths are geodesics with respect to the $\ell_\infty$ metric. 
We also observe that if $u\in [v\nea x]$ then $[v\nea u]\cup[u\nea x] = [v\nea x]$, and similarly for $u\in [v\sea x]$.

Therefore, we consider the $\ell_\infty$ metric on $X$, and denote as usual: 
$$d(v,x)= \text{ the length of the normal path  }[v\nea x].$$

For a set of vertices $A\subseteq X^{0}$, and a vertex $x\in X^{0}$, $$d(A,x):=\min_{a\in A}d(a,x).$$ 
We denote the 1-neighbourhood of $A$ by 
$$\calN_1(A)=\{x\in X^{0}\ | \ d(A,1)=1 \}$$
Also, 
for $v\in X^{0}$, $n\in\bbN$ we denote the $n$-th sphere, and $n$-th ball around $v$ respectively by 
\begin{align*}
\calS_n(v)&=\{x\in X^{0}\ | \ d(v,x)=n  \}\\
\calB_n(v)&=\{x\in X^{0}\ | \ d(v,x)\leq n \}
\end{align*}

Where  by abuse of notation, we denote by $\calN_1(A), \calS_n(v), \calB_n(v)$ also the corresponding induced subcomplex, that is, a higher dimensional cube is contained in this set if its vertices are contained in it. 

\subsection{The structure of spheres in CAT(0) square complexes}

Throughout the rest of this section, let $X$ be a two-dimensional CAT(0) square complex and 
$v\in X^{0}$. 
\begin{definition} \cite[Definition 1.7]{  ballmann1994polygonal}
    A vertex $x\in \calS_n(v)$ is called \textit{partly free} if: 
    \begin{itemize}
        \item There is exactly one edge $e$ incident to $x$, connecting it to some vertex $y\in \calS_{n-1}(v)$, and
        \item for any other edge $f$ incident to $x$, connecting it to a vertex $z\in \calS_n(v)$, there is a unique square $P$ with $f,e\in \partial P$. Equivalently, $\Lk(x,\calB_n(v))=\calN_1(\epsilon)$ where $\epsilon$ is the vertex in the link corresponding to the edge $e$ in the complex.
    \end{itemize}

    A vertex $x\in \calS_n(v)$ is called \textit{free} if it is not partly free, and there is exactly one square $P\in \calB_n(v)$ incident to $x$, with the two edges  in $\calS_n(v)$ incident to $x$, are in $\partial P$. Equivalently, $\Lk(x,\calB_n(v))$ is an edge. 
\end{definition}

    \begin{figure}
        \centering
        \includegraphics[]{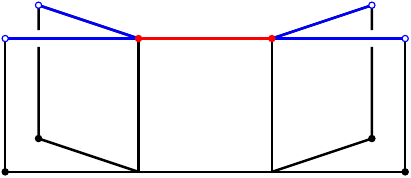}
        \caption[Free and partly free vertices]{The partly free vertices are shown in full circles, while free ones are empty circles. 
        The vertices of a block are shown in red, and its extension is obtained by adding the blue vertices.}
        \label{fig: partly free vertices}
        \label{fig: free vertices}
    \end{figure}

 See \Cref{fig: partly free vertices,fig: free vertices} for example.

\begin{remark}\label{the neighbours of a free vertex are partly free}
    By these definitions if $x\in \calS_n(v)$ is free, then it has exactly two neighbours in $\calS_n(v)$ , and they are both partly free. 
\end{remark}

\begin{remark} 
By \cite[Theorem 1.5 and Remark 1.9]{ballmann1994polygonal} all vertices in any sphere are either free or partly free. 
\end{remark}

The next lemma relates this definition to normal paths.
    \begin{lemma}\label{normal paths and pf}
        Let $x\in \calS_n(v)$, and let $y$ be the penultimate vertex in $[v\nea x]$. 
        If $x$ is partly free then $y$ is the unique neighbour of $x$ in $\calS_{n-1}(v)$ (i.e., $C_n$ is the unique edge incident to $x$ connecting it to $\calS_{n-1}(v)$), and if $x$ is free, then $y$ is the unique vertex in $\calS_{n-1}(v)\cap\calB_1(x)$ (i.e., $C_n$ is the unique square in $\calB_n(v)$ incident to $x$).
    \end{lemma}

    \begin{proof}
        By induction on $n$. Let $\{v_i\}_{i=0}^n$ and $\{C_i\}_{i=0}^n$ be the vertex and cube normal paths from $v$ to $x$. So $v_{n-1} = y$.
        Since $[v\nea x]$ is an $\ell_\infty$ geodesic, we have $y\in \calS_1(x)\cap \calS_{n-1}(v)$.
        This proves the case that $x$ is free and also the case $n=1$. 
        
        For $n>1$, if $x$ is partly free, let $y'$ be its neighbour in $\calS_{n-1}(v)$. Let us show that $y=y'$: The vertices $x,y,y'$ belong to $C_n$.
        By the induction hypothesis it is not difficult to see that  $\str(C_{n-1})$ must contain both $y,y'$. Since $\{C_i\}_{i=1}^n$ is a normal path we have $\{y,y'\}\subseteq \str(C_{n-1})\cap C_n = \{y\}$, and so $y=y'$ as desired.
    \end{proof}

    \begin{claim}\label{no 3 consecutive} 
    Assume $X$ is a CAT(0) square complex in which the shortest cycle in the link at each vertex is at least five, 
    then there cannot be three consecutive partly free vertices in the same sphere.    
    \end{claim} 

    \begin{proof}
        For $n=1$, any partly free vertex $x\in \calS_1(v)$ is connected by an edge $e_x$ to $v$. If there were two consecutive partly free vertices, say $x,y$, then $\{\{x,y\},e_x,e_y \}$ form a triangle in our square complex. So there are no two consecutive partly free vertices in $\calS_1(v)$, and in particular, there can not be three. 
     
        Let $n\geq 2$, and  assume to the contrary. Let $x\sim y\sim z$ be three consecutive partly free vertices in $\calS_n(v)$. 
        Here $x\sim y$ mean that $x,y$ are adjacent vertices in the sphere, and in particular, in a square. Because all $x,y,z$ are partly free, they are connected by an edge to a vertex in $\calS_{n-1}(v)$, denote these vertices by $x',y',z'$ respectively. 
        
        We now have the following squares: 
        $P_1=\{x',x,y,y'\}$ and $P_2=\{y',y,z,z'\}$ where $P_1,P_2$ share a common edge $\{y,y'\}$.

        If $y'\in \calS_{n-1}(v)$ is partly free, then from the definition, there is an edge connecting it to a vertex $y''\in \calS_{n-2}(v)$ (we are fine with the case that $n-2=0$, and $y''=v$).
        Since $x'\sim y'\sim z'$, and $y'$ is partly free, this gives us two more squares $P_3$ and $P_4$ with 
        $\{x',y'\},\{y',y''\}\in \partial P_3$ and $\{z',y'\},\{y',y''\}\in \partial P_4$. In particular, $\{ P_1,P_2,P_3,P_4\}$ are all incident to $y'$ and correspond to a 4-cycle in   $\Lk(y',X)$. This is a contradiction to our assumption, so $y'$ can not be partly free.

        If $y'$ if free, then  by \Cref{the neighbours of a free vertex are partly free} we have that $x',z'\in \calS_{n-1}(v)$ are partly free. But if $y'$ is free, and  $x',z'$ are its unique neighbours in $\calS_{n-1}(v)$, then there is a square $P$, such that $\{x',y'\},\{y',z'\}\in \partial P$. We now have the set $\{P_1,P_2,P\}$ of three squares, all incident to $y'$, where every two of them share a common edge. This correspond to a triangle in  $\Lk(y',X)$, contradiction. 
    \end{proof}




\subsection{Sectors and blocks}

    \begin{definition}
        Define a \textit{(partly free)  block}  in $\calS_n(v)$ to be a  connected component of the subgraph of $\calS_n(v)$ consisting only of partly free vertices.

        We define an \textit{extended block} in $\calS_n(v)$ to be a  block together with all its  adjacent free vertices in $\calS_n(v)$.
    \end{definition}

    \Cref{fig: free vertices} shows an example a block and its extended block.

For $v\in X^{0}$, $n\in\bbN$, denote by $P_n(v)$ the set of partly free vertices in $\calS_n(v)$.
\begin{definition}
    For any vertex $a\in \calS_1(v)$, define the \textit{$n$-th sector of $a$} to be $\calS_n(v,a):=\{x\in P_n(v)\ |\  a\in [v\sea x]  \}$. 
\end{definition}

That is, the sector $\calS_n(v,a)$ is the set of partly free vertices in $\calS_n(v)$ that one reaches from $v$ by first passing through $a$ in a normal path $[v\sea a]$.

\begin{claim}\label{blockes are not diveded by sectors}
    Assume $X$ is a CAT(0) square complex in which  
    the shortest cycle in the link at each vertex is at least five,
    then for every block $b\subseteq \calS_n(v)$ there is a unique $a\in\calS_1(v)$ such that $b\subseteq\calS_n(v,a)$.
\end{claim}

\begin{proof}
    Uniqueness is clear since different sectors are disjoint.
    
    We prove existence by induction on $n$.

    Let $b\subseteq \calS_n(v)$ be a block. 
    By \Cref{no 3 consecutive}, $|b|\leq 2$.
    If $b=\{x\}$, just take $a\in\calS_1(v)\cap [v\sea x]$.
    This also shows the base case $n=1$, as all blocks in $\calS_1(v)$ have size one. 
    
    Now, assume $n>1$ and $b=\{x,y\}$ for $x\ne y$.
    Let $x',y'$ be the neighbouring vertices to $x,y$ in $\calS_{n-1}(v)$.

    Observe that $\calS_{n-1}(v)\cap \calB_1(x) =  \{x',y'\}$, and thus $[v\sea x] \cap \{x',y'\}\ne \emptyset$.
    Similarly, $[v\sea y] \cap \{x',y'\}\ne \emptyset$.

    \textbf{Case 1.} If $x',y'$ are partly free then they form a block in $\calS_{n-1}(v)$, and by induction  there exists $a\in \calS_1(v)$ such that $\{x',y'\}$ is contained in the sector $\calS_{n-1}(v,a)$ and so $\{x,y\}$ are contained in the sector $\calS_n(v,a)$.
    
    \textbf{Case 2.} One of $x',y'$ is free --- say $y'$. We will show that $x'\in [v\sea x] \cap [v\sea y]$ and therefore $b\subseteq \calS_n(v,a)$ for $a\in \calS_1(v)\cap [v\sea x']$.
    
    Since $y'$ is free, $x'$ is partly free. Let $x''$ be the neighbour of $x'$ in $\calS_{n-2}(v)$. See \Cref{fig:blocks_in_proof}.
    Necessarily, $\calB_1(y') \cap \calS_{n-2}(v) = \{x''\}$, and so $x''\in [v\sea y']$.
    The vertex $x' \in [v\sea y]$, as otherwise $y'\in [v\sea y]$ but the path $y,y',x''$ is not a normal path.
    
    A similar argument shows that $x'\in [v\sea x]$.
\begin{figure}
    \centering
    \includegraphics[]{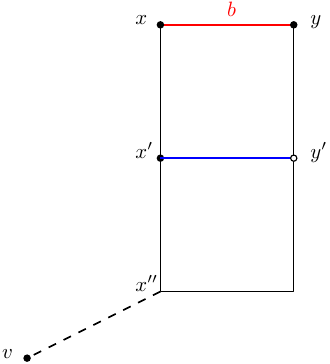}
    \caption{The block $b$ and its neighbours in Case 2.}
    \label{fig:blocks_in_proof}
\end{figure}
\end{proof}

\section{Davis complexes and universal groups}\label{the davis complex}

\subsection{The Davis complex} 
In this subsection we will define the Davis complex $X_L$, and describe its geometry.

\begin{definition}
For a finite graph $L=(V,E)$, define  the corresponding \emph{right-angled Coxeter group} by $$W_L=RACG (L) := \gen{  V(L) \  \middle| \begin{matrix} \  a^2,\  \forall a\in V(L),\\ \ [a,b],\  \forall a\sim b \end{matrix}} $$
where $a\sim b$ if $a$ and $b$ are neighbours in $L$.
\end{definition}

\begin{definition}[Davis \cite{davis1998cohomology}]\label{Davis complex}
The (right-angled) \emph{Davis complex} $X_L$ associated to $L$ is the CAT(0) cube complexes constructed in the following way:
Consider the Cayley graph $A_L$ associated to $W_L$ where we collapse the bigons, corresponding to the generators, to an edge (all generators are of order 2). Then,     fill in a cube whenever the 1-skeleton of one appears in $A_L$. 
\end{definition}

The link of a vertex $v\in X_L^{0}$ is the flag completion of the graph $L$ (i.e., the minimal flag simplicial complex containing $L$). As we will only be interested in two-dimensional CAT(0) cube complexes, let us explain this for the case that $L$ is triangle free.  
Indeed, every vertex in $L$ corresponds to a unique edge incident to $v$ (by the definition of Cayley graph and $W_L$) which corresponds to a unique vertex in $\Lk(v,X_L)$, and vice versa. 
For the edges, an edge $e\in E(L) $ between $a$ and $b$ in $V(L)$ corresponds to the relation $[a,b]$ in the presentation $W_L$, this corresponds to a 4-cycle in the Cayley graph with the edges labelled  $\{a,b,a,b\}$ ($a^{-1}=a$), which we fill in to get $X_L$, therefore a corresponding edge appears in  $\Lk(v,X_L)$. For the other way around, an edge in the link corresponds to a filled square, which is every 4-cycle, and the only 4-cycles are the ones of the form $\{a,b,a,b\}$, which corresponds to an edge in $L$. 

\begin{definition} \label{identification between link and L}
For any $v\in X_L^0$, denote by $c_v:\Lk(v,X_L) \to L$ the identification discussed above.
\end{definition}

\begin{remark} 
If $L$ has $m$ vertices and no edges, then $X_L$ is the $m$-regular tree.
\end{remark}

    To every hyperplane $\calH$ in $X_L$ we can assign the label $a\in V(L)$ of an edge $e\in X_L^1$ transverse to $\calH$ in $X_L$.
    The label of $\calH$ is well defined since the edges transverse to $\calH$ are the edges in the parallelism class of $e$ (i.e., the classes of the relation generated by being parallel edges in the same cube) and by construction, all the edges in a parallelism class have the same label. 
    
    If the label of $\calH$ is $a\in V(L)$ then $\calH$ is isomorphic to the Davis complex associated with the link $\Lk(a,L)$ of $a$ in $L$. Its carrier $N(\calH)$ --- that is, the minimal subcomplex which contains $\calH$ --- is isomorphic to the Davis complex associated with the star $\str(a,L)$ of $a$ in $L$.

\subsection{The universal group and local actions}\label{the universal group}

The following  definition was introduced in \cite{burger2000groups} for a regular tree $X=T_m$, we generalize it to Davis complexes. 

\begin{definition}[Local actions and universal groups] \label{def: universal group}
Let $L$ be a finite graph, and let $X_L$ be its associated Davis complex. For $g\in \Aut(X_L)$ and $v\in X_L^0$  denote by $\bfd_vg:\Lk(v,X_L) \to \Lk(gv,X_L)$ the isomorphism $g$ induces on the link at $v$.
\emph{The local action of $g$ at $v$} is the automorphism $$\bfD_vg= c_{gv}\circ \bfd_vg \circ c_v\ii \in \Aut(L).$$

For a subgroup $F\le \Aut(L)$, the \emph{universal group} of $F$ is defined as the subgroup 
$$U(F) = \{g\in \Aut(X_L) \;|\; \forall v\in X_L^0,\;\bfD_vg \in F\}$$
of all automorphisms whose local actions belong to $F$.
\end{definition}

\begin{remark}
The group $U(F)$ is a closed subgroup of $\Aut(X_L)$.
It is easy to see that $U(1) = W_L$ and $U(\Aut(L))=\Aut(X_L)$.
\end{remark}

\begin{definition}\label{def local action of group}
Let $G\le \Aut(X_L)$ act transitively on the vertices of $X_L$ and let $F\le \Aut(L)$. 
The \emph{local action of $G$ is $F$},
if $G\le U(F)$ and the image $\bfD_v(G_v)$ of the stabilizer $G_v$ is equal to $F$.
\end{definition}

 \begin{notation} For a complex $Y$ and a group action $H\actson Y$ we denote by 
 \begin{itemize}
    \item[] $H_v$ or $\Stab_H(v)$  the stabilizer in $H$ of the vertex $v\in Y$;  
    \item[] $\Stab_H(A)$ the set-wise stabilizer of a subset $A\subseteq Y$;
    \item[]$H_A$ or $\Fix_H(A)$ the fixator (i.e., point-wise stabilizer) of the subset $A\subseteq Y$; 
    \item[]$H^n_v$ the fixator of $\calB_n(v)$; 
    \item[] and $H^n_A$ the fixator of $\calN_n(A)$.
 \end{itemize}
    When $H$ stabilizes a set $A$ (set-wise), we denote by $H|_A$ the subgroup of $\Aut(A)$ which is obtained by restricting the elements of $H$ to the set $A$.
\end{notation} 

\begin{remark}
    We note that the restriction to a ball of radius 1 is determined by the induced map on the link. In particular, $H_v|_{\calB_1(v)} \simeq \bfD_v (H_v)$.
\end{remark}

\subsection{The Odd graph.}

We begin by recalling the definition of the Odd graph $O_d$.
\begin{definition}[{Biggs-Guy \cite{biggs1972edge}}] 
For an integer $d\geq 1$, the Odd graph $O_d$ is the graph whose  vertex set is $\binom{[2d-1]}{d-1}$, and edge set is 
$\{ \{A,B\} |  A\cap B=\emptyset \}$. That is, there is an edge between two vertices if and only if their associated subsets do not intersect. 
\end{definition}

Perhaps the most famous Odd graph is the Petersen graph $O_3$, however the Odd graphs we will consider later will be for higher $d$. Namely, we will consider $d\geq 6$ because we will want to use the fact that $\Alt_{d-1}$ is simple. 

\begin{remark}\label{girth of Odd graph}
    For $d\geq 4$, the girth of $O_d$, i.e., the shortest cycle, is 6. 
\end{remark}

Note that $\Sym_{2d-1}\le \Aut(O_d)$. In fact, by the Erd\"os-Ko-Rado theorem 
\cite{erdos1961intersection} we have:
\begin{theorem}[{\cite[Corollary 7.8.2]{godsil2001algebraic}}]
    $\Aut(O_d)=\Sym_{2d-1}$ .
\end{theorem}

In view of the above theorem we observe the following:
\begin{observation}\label{fix in the Odd graph}
A vertex is fixed by an element in $\Aut(O_d)$ if and only if its associated set is stabilized by the corresponding element in $\Sym_{2d-1}$.
Therefore, if  $A\in\binom{[2d-1]}{d-1}$ is a vertex in $O_d$, and $e$ is an edge in $O_d$ with endpoints $B,C\in\binom{[2d-1]}{d-1}$, then:
    \begin{enumerate}
        \item \label{fix of vertex} 
        $\Stab_{\Aut(O_d)}(A)=\Sym(A)\times\Sym([2d-1]\setminus A).$
        \item \label{fix of edge} 
        $\Fix_{\Aut(O_d)}(e)=  \Sym(B)\times\Sym(C).$
        \item \label{fix of star} $\Fix_{\Aut(O_d)}(\calN_1(A))=  \Sym(A)$.
        \item \label{fix of 1-nbd of edge} $\Fix_{\Aut(O_d)}(\calN_1(e))=1$.
    \end{enumerate}
\end{observation}


\subsection{The stabilizer of a vertex in $U=U(\Alt_{2d-1})$.}

From now on let $d\geq 6$, denote $\ell=2d-1$, $m=d-1$, $L=O_d$, $X=X_L$ and $U=U(\Alt_\ell)$ where $\Alt_\ell$ is the alternating group in $\Sym_\ell = \Aut(L)$.

We start with the following claim, which we will use mostly to determine that the action of two non adjacent partly free vertices in the same square, determines 
the action of the  free vertex between them.

\begin{claim}\label{auxiliary claim for the action of partly free vertices implies the  free one}
    Let $G \le \Aut(X)$ and let $x,y,z$ be consecutive vertices in a square of $X$, then 
    $G^1_x\cap G^1_z\subseteq G^1_y$.
\end{claim}
\begin{proof}
    Let $P\in X^2$ be the square which contains $x,y,z$. 
    Let $g\in G^1_x\cap G^1_z$, so $g$ fixes ${\calB_1} (x)\cup\calB_1(z)$, 
    and hence, $g$ fixes $P$.
    Moreover, if $\epsilon_x$ is the vertex in $\Lk(y,X)$ that corresponds to the edge  $e_x$ connecting $x$ to $y$, and $\epsilon_z$ is the vertex in $\Lk(y,X)$ that corresponds to the edge $e_z$ connecting $z$ to $y$, then $\bfD_yg\in \Fix_{\Aut(L)}(\calB_1(\epsilon_x)\cup \calB_1(\epsilon_z))$. 
    Now let $\sigma$ be the edge in $\Lk(y,X)$ that corresponds to the square $P$,   so $\calN_1(\sigma)= \calB_1(\epsilon_x)\cup\calB_1(\epsilon_z)$,  
    therefore $\bfD _y g\in \Fix_{\Aut(L)}(\calN_1(\sigma))$.
    By \Cref{fix in the Odd graph} the fixator of the one neighbourhood of an edge in $\Lk(y,X)$ is trivial, therefore $\bfD_y g=1$, hence $g\in \Fix(\calB_1(y))$.
\end{proof}

\begin{observation}\label{link to complex automorphism}
    For every vertex $v\in X_L^0$ and automorphism $\phi\in\Aut(L)$, there exists a (unique) automorphism $\Phi_{v,\phi}\in\Aut(X_L)$ such that $\Phi_{v,\phi}(v)=v$ and $\bfD_x \Phi_{v,\phi} = \phi$ for all $x\in X_L$. To see this, we note that $\phi$ gives rise to an automorphism of the group $W_L$, sending generators to generators. Thus, it gives rise to an automorphisms of the Cayley graph $A_L$, and the Davis complex $X_L$. This automorphism fixes the vertex corresponding to $1\in W_L$ and the local action is $\phi$ at each vertex. By conjugating this automorphism by an element of $W_L$ we can arrange for it to fix any other vertex $v\in X_L$.
\end{observation}

The next proposition is a special case of \cite[Theorem 4.2]{lazarovich2018regular}, and we add it with a proof in the language of this paper, for completeness.

\begin{proposition}\label{U_v^n is a product}
Let $v$ be a vertex of $X$.
Denote by $Bl_n(v)$ the set of all blocks in $\calS_{n}(v)$.
    \begin{enumerate}[label=(\arabic*)]
        \item For every $x\in b\in Bl_n(v)$ the map $\bfD_x$ gives an isomorphism
    $$ U_v^n|_{\calN_{1}(b)}\xrightarrow{\simeq}\Alt(Y_b).$$ 
    for some $Y_b\in V(L)$.
        \item The natural restriction map gives the following isomorphism $$U_v^n|_{\calB_{n+1}(v)} \xrightarrow{\simeq} \prod_{b\in Bl_n(v)} U_v^n|_{\calN_{1}(b)}.$$
        \item If $y\in \calS_n(v)$ is free, then $$U^n_v|_{\calN_1(y)} \xleftarrow{\simeq} U^n_v|_{\calN_1(x)} \times U^n_v|_{\calN_1(z)}$$ where $x,z$ are the two neighbours of $y$ in $\calS_n(v)$.
    \end{enumerate}
\end{proposition}
\begin{proof}
    (1)
     Note that by \cref{no 3 consecutive,girth of Odd graph}, each block has size at most 2. Consider a block $b$, each vertex of $b$ has an edge connecting it to a vertex in $\calS_{n-1}(v)$. These edges transverse the same hyperplane $\calH_b$, the unique hyperplane separating $b$ from $\calB_{n-1}(v)$. The hyperplane $\calH_b$ is  labelled by some $Y_b\in V(L)$.
     Let $h_b$ be the set of vertices of the halfspace of $\calH_b$ that contains $b$ and let $h^*_b$ be the set of vertices of the complementary halfspace.

    Let $x\in b$. Every element $g\in U^n_v$ satisfies that $\bfD_x g$ fixes the 1-neighbourhood of the vertex labelled $Y_b$ in $\Lk(x,X_L)$, and so by \Cref{fix in the Odd graph}, $\bfD_x g\in \Alt(Y_b)$. 
    
    Now, if $b$ has another vertex $x\ne y\in b$, then $x,y$ belong to a square $\{x,y,y',x'\}$ where $x',y'\in \calS_{n-1}(v)$. If $g\in U^n_v$, then $\bfD_{y'}
    g=\id$ (since $\calN_1(y')\subseteq\calB_n(v)$). If $g$ is in the kernel of the map $\bfD_x$, i.e., $\bfD_x g= \id$ then by \Cref{auxiliary claim for the action of partly free vertices implies the  free one} $\bfD_y g= \id$, and so $g|_{\calN_1(b)}=\id$. This shows that the map $\bfD_x:U_v^n|_{\calN_{1}(b)}\xrightarrow{\simeq}\Alt(Y_b)$ is injective. (In fact, one has $\bfD_x g=\bfD_yg$).

    To prove surjectivity, recall that by \Cref{link to complex automorphism} for every $\phi\in\Alt_\ell$, there is a unique automorphism $\Phi\in U$ with $\Phi(x)=x$ and $\bfD_w\Phi = \phi$ for all $w$. 
    We note that if $x\in b$ and $\phi\in \Alt(Y_b)$ then $\Phi$ fixes the hyperplane $\calH_b$ pointwise, and also fixes $b$ pointwise. 
    We can now define a new automorphism 
    $$\Phi'(w) = \begin{cases} \Phi(w) & w\in \calN_1(h_b)\\
    \id & w\in \calN_1(h_b^*)\end{cases}$$
    The automorphism $\Phi'$ satisfies $\bfD_x\Phi' = \phi$. We note that $B_n(v)\subseteq \calN_1(h_b^*)$.
    This shows that $\Phi'\in U^n_v$ and its image under the map $\bfD_x$ is $\phi.$ Thus the map $\bfD_x : U^n_v|_{\calN_1(b)} \to \Alt(Y_b)$ is surjective.

    (2)
    To prove injectivity, note that if $g\in U^n_v$ satisfies $g|_{\calN_1(b)}=\id$ for every block $b\in Bl_n(v)$ then 
    \Cref{auxiliary claim for the action of partly free vertices implies the  free one} shows that $g|_{\calN_1(y)}=\id$ for every free vertex $y\in \calS_{n}(v)$. It follows that the map $$U_v^n|_{\calB_{n+1}(v)} \hookrightarrow \prod_{b\in Bl_n(v)} U_v^n|_{\calN_{1}(b)}$$ is injective.
    
    To prove surjectivity, 
    let $b_1,\ldots, b_{t}$ be the blocks in $Bl_n(v)$, and let $\phi_i\in \Alt(Y_{b_i})$ be any choice of permutations for each $1\le i\le t$. 
    Our goal is to find an automorphism $\Phi$ such that $\Phi\in U^n_v$ and $\bfD_x\Phi = \phi _{i}$ for all $x\in b_i$, $i=1,\dots,t$.
    As in the proof of (1), there exists an automorphism $\Phi'_i$ such that $\bfD_x\Phi'_i = \phi_i$ for $x\in b_i$, and $\Phi'_i=\id$ on $\calN_1(h^*_{b_i})$. 
    Since $\calB_n(v) \subseteq \calN_1(h^*_{b_i})$, we get $\Phi'_i \in U^n_v$. 
    For $j\ne i$, since $\calN_1(b_j)\subseteq \calN_1(h^*_{b_i})$, we get $\bfD_x\Phi'_i = \id$ for $x\in b_j$.
    Finally, the product $\Phi = \Phi_1\circ \dots \circ \Phi_t$ is the desired automorphism.
    

    

    (3) 
    If  $g\in U_v^n$ satisfies that $g|_{\calN_1(x)} =g|_{\calN_1(z)}=\id $ then by \Cref{auxiliary claim for the action of partly free vertices implies the  free one} $g\in U_v^n|_{\calN_1(y)}=\id$.
    This shows that the homomorphism $$U^n_v|_{\calN_1(x)} \times U^n_v|_{\calN_1(z)} \to U^n_v|_{\calN_1(y)} $$ is well-defined.
    It is also clearly surjective.
    To show that it is an isomorphism, we note that by (1) we have $U^n_v|_{\calN_1(x)} \simeq U^n_v|_{\calN_1(z)} \simeq \Alt_m$. Also, $U^n_v|_{\calN_1(y)}\simeq \bfD_y(U^n_v)$ which is a subgroup of the stabilizer of an edge in the link $L$ (namely, the edge that corresponds to the unique square of $\calB_n(v)$ incident to $y$). Such an edge stabilizer is isomorphic to $\Alt_m\times \Alt_m$ by \Cref{fix in the Odd graph}.
    By order considerations, it follows that the map is an isomorphism.
\end{proof}

 

    
We end this subsection with a technical lemma that will be used in the next section.

\begin{lemma}\label{auxiliary claim for components in universal} 
    Let $v$ be a vertex of $X$, and let $x,y\in\calS_n(v)$ be partly free vertices in different sectors. Then, there exists $\Phi\in U_v^{n-1} $ such that either $\Phi(x)\ne x$ and $\Phi(y)=y$, or $\Phi(x)= x$ and $\Phi(y)\ne y$.
\end{lemma}

\begin{proof}
    For $n=1$, $x,y$ are partly free, therefore they are connected by an edge to $v$. Each edge corresponds to a vertex $Y_x,Y_y\in \Lk(v,X)$ respectively.  
    Since $U_v|_{\calS_1(v)}= \Alt_{\ell}$, take any $\phi\in \Alt_\ell$ such that $\phi(Y_x)=Y_x$ and $\phi(Y_y)\ne Y_y$. 
    Then extend $\phi$ to $\Phi_{v,\phi}$ according to \Cref{link to complex automorphism}.
    
    Let $n\geq 2$ 
    Let $x,y\in \calS_n(v)$ be two partly free vertices in different sectors. Each connected by an edge to $\calS_{n-1}(v)$, denote the end points of those edges by $x'\ne y'$ respectively. 

    We claim that there exists a hyperplane $\calH$ separating $x'$ and $y'$ and not intersecting $\calB_{n-2}(v)$:
    If $x'$ is free, then there are two hyperplanes $\calH_1,\calH_2$ separating $x'$ from $\calB_{n-2}(v)$. If we denote by $h_1,h_2$ the halfspaces bounded by $\calH_1,\calH_2$ respectively containing $x'$, then $h_1\cap h_2 \cap \calS_{n-1}(v)=\{x'\}$.
    Take $\calH$ to be one of $\calH_1,\calH_2$ that separates $x'$ from $y'$.
    Same argument works if $y'$ is free.
    Now, if $x',y'$ are partly free, then they belong to different blocks (otherwise by arguments similar to those of \Cref{blockes are not diveded by sectors}, $x,y$ belong to the same sector). 
    Let $\calH$ be the unique hyperplane separating $x'$ from $\calB_{n-2}(v)$ and let $x'\in h$ be the halfspace bounded by $\calH$, then $h\cap \calS_{n-1}(v)$ is the extended block containing $x'$. Therefore, $y'$ is not in $h$.
    
    Let $\calH$ be a hyperplane as in the previous paragraph, and let $h,h^*$ be the halfspaces bounded by $\calH$. Without loss of generality assume that
    $x'\in h$ and $\calB_{n-2}(v)\cup\{y'\}\subseteq  h^*$.
    Let $A\subseteq [\ell]$ be the label of the hyperplane $\calH$,
    and let $B\subseteq [\ell]$ be the label of the edge $e$ connecting $x,x'$.
    Since $x$ is partly free, the edge $e$ is not in the carrier of $\calH$.
    It follows that $A$ and $B$ are not connected by an edge in $L$, i.e., $A\cap B \ne \emptyset$.
    Therefore, there exists some alternating permutation $\phi\in \Alt(A)\le \Alt_\ell$ such that $\phi(B)\ne B$.
    
    Take $\Phi_{x',\phi}$ 
    to be the
    automorphism such that  $\Phi_{x',\phi}(x')=x'$ and $\bfD_{z}\Phi_{x',\phi}= \phi$ for all $z\in X^0$ given by \Cref{link to complex automorphism}. Now,  define $$\Phi(z) = \begin{cases}
       \Phi_{x',\phi}(z) & z\in h\\
    \id & z\in h^*,
    \end{cases}$$          
    then $\Phi(x)\ne x$ since $\phi(B)\ne B$, and $\Phi(y)=y$ since $y\in h^*$.
    \end{proof}

\subsection{The virtual simplicity of the universal group}


\begin{definition}
    The group $\Aut(X)^+\leq \Aut(X)$ is the subgroup generated by all the halfspace stabilizers. $U^+\leq U$ is the subgroup generated by all elements in $U$ that also stabilize a 
    halfspace, i.e.,
    $U^+=U\cap \Aut(X)^+$.
\end{definition}


\begin{claim}\label{U^+ is index two in U}
    $|U:U^+|=2$.
\end{claim}

\begin{proof}
First note that $U^+$ preserves the proper 2-colouring $C:X^0\to \{\text{red,blue}\}$.
Therefore, we are left to show that the orbit of a vertex $v$ is the whole set of vertices coloured in the same colour as $v$, and that $U_v=U_v^+$.

\textbf{Showing that the orbit is the whole set.} 
To do that, we first show that for any $v\in X^0$, $U^+_v|_{\calB_1(v)}$ acts transitively on the 
neighbours of  $v$:  
Let $Y\in V(L)$. Any $\phi\in \Fix_{\Alt_\ell}(\calN_1(Y))\le\Aut(L) $ 
can be extended to $\Phi\in \Aut(X)$ by \Cref{link to complex automorphism}.
Fixing the 1-neighbourhood of a vertex in $L$ corresponds to fixing the star of the edge  $Y_e$ associated with $Y$ in $X$, in particular it fixes the hyperplane transverse to  $Y_e$, 
therefore $\Phi\in U^+$.

For each $a\in V(L)$, $a\in \binom{[\ell]}{m}$, by \Cref{fix in the Odd graph}, $\Fix_{\Alt_\ell}(\calN_1(a))=\Alt(a) \simeq \Alt_{m}$. Going over all elements in $V(L)$ we have  $\gen {\Fix_{\Alt_\ell}(\calN_1(a))|a \in V(L)}=\gen {\Alt(a)|a\in \binom{[\ell]}{m}}=\Alt_{\ell}$. 

Let $x,y\in X^0\cap\calN_1(v)$, let $e_x,e_y\in X^1$ be the edges connecting $x$ and $y$ to $v$, and let $a_x,a_y\in \binom{[\ell]}{m}$ by the sets in $L$ associated with $e_x,e_y$ respectively. Then there is a permutation $\phi\in\Alt_\ell$ such that $\phi(a_x)=a_y$. Take $\Phi\in U_v$ as described above, i.e., with local actions $\phi$,  then  $\Phi\in U_v^+$ and $\Phi(x)=y$. 

Next we use transitivity in the 1-neighbourhood, to show that $U^+$ is transitive on the colour class.  

    Let $v,x\in X^0$ be two vertices with $C(x)=C(v)$. Let $P$ be any path along edges between $v$ and $x$, for convenience assume its without backtracking. Denote the vertices of the path by $v=b_1,r_1,b_2,r_2,\ldots b_k,r_k,b_{k+1}=x$, where $b$ is for vertices coloured 'blue', and $r$ is for vertices coloured 'red'. For every $1\leq i\leq k$, there is an element in $\Phi\in U^+_{r_i}$ such that 
    $\Phi(b_i)=b_{i+1}$. Hence, the orbit of $v$ under $U^+$ is its all colour class.

    \textbf{Showing that $U_v=U_v^+$.} 
    Clearly $U_v^+\subseteq U_v$. For the other inclusion, let $g\in U_v$. 
    First step: showing we can assume that $g\in U_v^1$. If there is an element $h\in U_v^+$ such that $h|_{\calB_1(v)}=g|_{\calB_1(v)}$, then $gh\ii\in U_v^1$, and $gh\ii \in U_v^+$ if and only if $g\in U_v^+$. 
Let us show that there exists such $h$. Indeed, $g|_{\calS_1(v)}$ induces an automorphism on $\Lk(v,X)$, which is an element in $\Alt_\ell$. On the other hand, for an edge $e$ incident to $v$, any $\phi\in U_v$ fixing the hyperplane transverse to $e$ is in fact in $U_v^+$, and by \Cref{fix in the Odd graph}, the local action it induces on $\calS_1(v)$ is an element in $\Alt(Y_e)$, where $Y_e$ is the set corresponding to $e$ in $\Lk(v,X)$. In fact every element  $\phi \in \Alt(Y_e)$ can be extended to an element $\Phi'$ in $U_v^+$, similarly to the way $\Phi'$ was defined in \Cref{U_v^n is a product}.
This is true for any edge incident to $v$, therefore, $Y_e$ goes over all $m$ subsets of $[\ell]$, and $\gen{\{\Alt(Y_e)| Y_e\in\binom{[\ell]}{m} \}}=\Alt_\ell$, therefore $g|_{\calS_1(v)}$ can be presented by elements in $U_v^+|_{\calS_1(v)}$, and since any such permutation can be extended to an element in $U_v^+$, there is such $h\in U_v^+$ with $h|_{\calB_1(v)}=g|_{\calB_1(v)}$.

Assume $g\in U_v^1$. 
Since we consider $d\ge 6$ the girth of $O_d$ is 6. Choose two vertices $A,B\in V(L)$ such that $d_L(A,B)=3$ ($d_L$ refers to the distance in the usual sense for graphs). Take the two vertices $a,b\in\calS_1(v)$ such that the edge connecting $a$ to $v$ is labeled $A$, and similarly for $b$. 


Denote by 
$\calH_a,\calH_b$ the hyperplanes transverse to $A,B$ respectively. 
Let $h_{a}, h_{b}$ be the halfspaces of $\calH_a,\calH_b$ respectively, that contains $v$, and let $h_a^*, h_{b}^*$ be the complementary halfspaces.
Note that $(h_a^*\cup \calN_1(N(\calH_a)) )\setminus \calB_1(v) \subseteq h_b$ (in words, the halfspace $h_a^*$ and the 1-neighbourhood of the carrier of $\calH_a$, aside from $\calB_1(v)$, is contained in $h_b$).


Let us define an automorphism $\Phi\in U_v$ in the following way. For any vertex $x\in h_{b}^*$, let $\bfD_x\Phi=\id$. As for the vertices in $h_b$, 
note that is it enough to define $\Phi$ on the partly free vertices, due to \Cref{U_v^n is a product}. Let $x\in h_{b}$ be partly free. If it is incident to an edge transverse to $\calH_{b}$, define $\bfD_x\Phi=\id$. 
If it is a part of block of size 2, such that the second vertex is incident to an edge transverse to $\calH_{b}$,  then define here also $\bfD_x\Phi=\id$. 
For any other partly free vertex in $h_b$, define $\bfD_x\Phi=\bfD_xg$.

We now show that if $x$ is a partly free vertex in $h_a^*$, then $\bfD_x\Phi=\bfD_xg$. Indeed, 
in order for a partly free vertex in $h_a^*$ not to satisfy the condition $\bfD_x\Phi=\bfD_xg$, it needs to be in a block of size two, with the second vertex being in $h_b$, but that can not be the case from the choice of $a$ and $b$.

To show that 
$\Phi|_{h_a^*}=g|_{h_a^*}$, let $x\in h_a^* $ be a vertex. Consider $[v\nea x]=\{v,v_1,v_2,\ldots,v_n=x\}$. Since $g,\Phi\in U_v^1$, we have $g(v_1)=\Phi(v_1)=v_1$. For the next step, we note that $v_1\in h_a^*$,  therefore $\bfD_{v_1}g=\bfD_{v_1}\Phi$, and so $g(v_2)=\Phi(v_2)$. And so on for any other $2\leq i\leq n$, since $g(v_{i-1}) =\Phi(v_{i-1})$ and $\bfD_{v_{i-1}}g =\bfD_{v_{i-1}}\Phi$, we have that $g(v_{i}) =\Phi(v_{i})$, in particular $g(x) =\Phi(x)$ $\Rightarrow$ $g|_{h_a^*}= \Phi|_{h_a^*}$. 

By definition, we have that $\Phi|_{h_b^*}=\id$, therefore it fixes $\calH_b$, and also $v$. So $\Phi\in U_v^+$. 
In addition, $\Phi|_{h_a^*}=g|_{h_a^*}$, we have that $g\Phi\ii|_{h_a^*}=\id$, in particular, $g\Phi\ii$ fixes $\calH_a$, and also $v$. Therefore $g\Phi\ii\in U_v^+$. 
This implies that $g=g\Phi\ii\cdot \Phi\in U_v^+$.
\end{proof}



\begin{proposition}\label{simplicity of U}
    The group $U$ is a closed virtually simple subgroup of $\Aut(X)$.
\end{proposition}

\begin{proof}
    In \cite[Theorem A.3]{lazarovich2018regular}, Lazarovich proves that if the following holds: 
    \begin{itemize}
        \item $X$ is a proper finite dimensional irreducible CAT(0) cube complex, and
        \item $G\leq \Aut(X)$ is non-elementary, acting essentially on $X$ with property $(P)$, $\Lambda (G)=\partial X$ and $G$ acts faithfully on $\partial X$, 
    \end{itemize}
    then $G^+$ is simple or trivial. 
    
    In a similar way to the proof of \cite[Corollary 5.6]{lazarovich2018regular}, we have that $X$ and the universal group $U$ satisfy the required conditions. Since $U^+$ is not trivial, it is simple, together with \Cref{U^+ is index two in U} we are done.

    We refer the reader to \cite{lazarovich2018regular} for the relevant definitions and the proof of \cite[Corollary 5.6]{lazarovich2018regular}.  
\end{proof}


\section{Local to Global}\label{local to global section}

In \cite[Lemma 3.5.3]{burger2000groups} Burger-Mozes proved a ``local to global'' result for the universal group $U=U(\Alt_m)$, $m\ge 6$, of automorphisms of the $m$-regular tree. Their result states that for a group to be dense in $U$ it suffices to verify that it is vertex transitive, its local action is $\Alt_m$ and that it is non-discrete.  We dedicate this section to prove a version of this theorem for the universal group $U=U(\Alt_{2d-1})$ of automorphisms of $X_L$ where $L=O_d, d\ge 6$.

\subsection{Non-discreteness criterion}
In \cite[Proposition 3.3.2]{burger2000groups}, Burger and Mozes provide a local condition for non-discreteness of the group $G$ acting vertex-transitively on a $d$-regular tree $T_d$.
When the local action is $\Sym_d$ they prove that the restriction $G^1_v|_{\calB_2(v)}$ is either $1,\Sym_{d-1}$ or a direct product of $d$ copies of $\Sym_{d-1}$.
Moreover, the last case holds if and only if $G$ is non-discrete in $\Aut(T_d)$.
We prove a similar result:

\begin{proposition}\label{local action options}
    Let $6\le d \in \bbN$ and let $G\le \Aut(X_{O_d})$ be a vertex transitive subgroup with local action  $\Alt_{2d-1}$.
    Then, for all $v\in X_L^0$, $G_v^1|_{\calB_2(v)}$ is isomorphic to $1, \;\Alt_{d-1}$ or $(\Alt_{d-1})^{\binom{2d-1}{d-1}}$. 
    
    Moreover, if $G$ is non-discrete then $G^1_v|_{\calB_2(v)} = (\Alt_{d-1})^{\binom{2d-1}{d-1}}$.
\end{proposition}

\begin{remark}
\begin{itemize}
\item Since the group $G$ is vertex transitive, for all $v,w\in X^0_L$ we have $G_v^1|_{\calB_2(v)}\simeq G_w^1|_{\calB_2(w)}$.

\item We will see in \Cref{characterization of dense subgroups} that $G$ is non-discrete if and only if $G^1_v|_{\calB_2(v)} = (\Alt_{d-1})^{\binom{2d-1}{d-1}}$.
\end{itemize}
\end{remark}


For the proof we will need the notion of a ``component of a finite group'' \cite[Definition (7.1)]{borel1973homomorphismes}. We recall the definition here:
\begin{definition}
    A group $C$ is \emph{quasi-simple} if it is perfect (that is, $C=[C,C]$) and its inner automorphism group is simple (that is, $C/Z(C)$ is simple).
    A subgroup $C\le H$ is a \textit{component} of $H$ if it is quasi-simple, finite and subnormal in $H$.
\end{definition}

\begin{fact}\label{properties of components}
We gather some properties of components (for more details and proofs cf. \cite[Chapter 11, section 31]{aschbacher2000finite}): 
    \begin{enumerate}
        \item If $C\leq \prod_{i=1}^t S$ is a component, where $S$ is a simple group, then $C$ is one of the factors.  
        \item \label{component in quotients} If $C\leq H$ is a component, and $Q=H/N$ is a quotient of $H$, let $q:H\to Q$ the quotient map, then $q(C)$ is either a component in $H$ or trivial.
        (In particular, if $S$ is simple then the image of a component $C\leq \prod_{i=1}^t S$ under the projection to one of the factors is either $S$ or 1).
        \item \label{components commute} If $C,C'\leq H$ are distinct components, then $[C,C']=1$.
    \end{enumerate}
\end{fact}

Throughout this section we will make repeated use of the following easy fact:
\begin{fact}\label{fact: fix normal in stab}
    If $H\actson Z$, $A\subseteq Z$, and $H\leq \Stab(A)$ (set-wise stabilizer) then the point-wise stabilizer $\Fix_H(A)$ is normal in $H$.
\end{fact}

\begin{proof}
    Let $w$ be a neighbour of $v$. 
    Let $e$ be the edge connecting them, and let $H$ be the subgroup of $G$ of all automorphisms fixing the set of squares incident to $e$. 
    By \Cref{fix in the Odd graph}.\ref{fix of star} we have 
    $H|_{\calB_1(w)} \simeq \Alt_{2d-1}$.

    By \Cref{fact: fix normal in stab} $$G_v^1|_{\calB_1(w)} \normalin H|_{\calB_1(w)}\simeq \Alt_{d-1}$$
    Since $d\ge 6$ it follows that $G_v^1|_{\calB_1(w)}\simeq 1\text{ or }\Alt_{d-1}$.

    Note that each neighbour of $w$ is partly free, and is a block of size one. 
    By \Cref{U_v^n is a product}, 
    \begin{equation}\label{G^1_v in the product of Alts}
        G_v^1|_{\calB_2(v)} \le U^1_v|_{\calB_2(v)}\simeq \prod _{w\in Bl_1(v)} U^1_v|_{\calB_1(2)} \simeq \prod_{w\in Bl_1(v)}\Alt_{d-1}
    \end{equation} where the product is over the neighbours $w$ of $v$.
    
    If $G^1_v|_{\calB_2(v)} = 1$ then we are done, and moreover, $G$ is clearly discrete.
    
    Otherwise, $G^1_v|_{\calB_1(w_0)}\simeq \Alt_{d-1}$ for some neighbour $w_0$ of $v$.
    Since $G^1_v\normalin G_v$, if $g\in G_v$ sends the neighbour $w_0$ to $w$ then $G^1_v|_{\calB_1(w)} \simeq \Alt_{d-1}$ as well.
    
    The action of $\Alt_{2d-1}$ on the vertices of $L$ is primitive. Since, by assumption, $G_v|_{\calB_1(v)} \simeq \Alt_{2d-1}$, we get that $G_v$ acts primitively (and therefore transitively) on the neighbours of $v$.
    It follows that for every neighbour $w$ of $v$ we have $G^1_v|_{\calB_1(w_0)}\simeq \Alt_{d-1}$.

    Consider a component $C$ of $G^1_v|_{\calB_2(v)}$.
    By \Cref{properties of components}.\ref{component in quotients}, $C|_{\calB_1(w)}\le G^1_v|_{\calB_1(w)} \simeq \Alt_{d-1}$ is either trivial or $\Alt_{d-1}$.
    Let $B_C$ be the collection of neighbours of $v$ for which $C|_{\calB_1(w)} \simeq \Alt_{d-1}$.
    By \Cref{properties of components}.\ref{components commute} if $C,C'$ are distinct components, then $B_C \cap B_{C'}$
    The action of $\Alt_{2d-1}$ on the vertices of $L$ (i.e., subsets of size $d-1$ in $[2d-1]$) is primitive. 
    Therefore, the sets $B_C$, for the components $C$ of $G^1_v|_{\calB_2(v)}$ form a partition of the neighbours of $v$.
    Conjugation by $g\in G_v$ sends a component $C$ to the component $gCg\ii$,  and we have $gB_C = B_{gCg\ii}$. Thus, the partition $B_C$ is invariant under the action of $G_v$ on the neighbours of $v$.
    Since this action is primitive this partition must be trivial. That is, either  
    (1) there is only one component $C$ and $B_C$ is the set of all neighbours of $v$, or (2) Each $B_C$ is a singleton. 

    In the first case, $C \simeq \Alt_{d-1}$, and we may think of $C$ as a diagonal copy of $\Alt_{d-1}$ in the product $\prod_{w\in Bl_1(v)} \Alt_{d-1}$. Every element $g$ of the product which is not in $C$ (i.e., in the diagonal) satisfies $gCg\ii \ne C$. Since $C$ is the only component in $G^1_v|_{\calB_2(v)}$, it follows that $G^1_v|_{\calB_2(v)} = C\simeq \Alt_{d-1}$.
    The same argument shows moreover that $G$ is necessarily discrete in this case. 
    
    In the second case, we get that each component $C$ is exactly one of the factors of the form $\Alt_{d-1}$ in the product in  \eqref{G^1_v in the product of Alts}.
    It follows that $G^1_v|_{\calB_2(v)} = U^1_v|_{\calB_2(v)} $, and so $G^1_v \simeq (\Alt_{d-1})^{\binom{2d-1}{d-1}}$.
\end{proof}

\subsection{Criterion for dense subgroups of universal groups}

In this section we prove the following.
\localToGlobal*


The rest of the section is devoted to the proof of this theorem.

\begin{proof}
First, we note that $(1) \implies (2)$ is trivial, and $(2) \implies (3)$ is proved in \Cref{local action options}. 
Thus, we focus on $(3)\implies (1)$:

We recall our standing notation $L=O_d,\;X=X_L,\;m=d-1,\; \ell=2d-1$.
Let $G \le U$ be a vertex transitive subgroup with 
$G_v|_{\calB_2(v)} = U_v|_{\calB_2(v)}.$

Since $G$ and $U$ are vertex transitive, it suffices to show that $\cl{G_v}=U_v$ for some fixed $v\in X^0$.
In order to show $\cl{G_v}=U$, we need to show that for all $n\in \bbN$,
$$G_v |_{\calB_n(v)} = U_v |_{\calB_n(v)}.$$
We do so by induction on $n$.

The base cases, $n=1$ and $n=2$, follow from the assumption.

For the induction step, assume $n\ge 2$ and $G_v |_{\calB_n(v)} = U_v |_{\calB_n(v)}.$
Our goal is to show that $$G_v |_{\calB_{n+1}(v)} = U_v |_{\calB_{n+1}(v)}.$$

To do so, we will show that $$G_v^n|_{\calB_{n+1}(v)}=U_v^n|_{\calB_{n+1}(v)}.$$ 
Then, the inclusion $G\le U$ gives maps between the following exact sequences
\[\begin{tikzcd}
1\arrow{r} &G_v^n|_{\calB_{n+1}(v)}\arrow{r}\arrow{d} &G_v|_{\calB_{n+1}(v)}\arrow{r} \arrow{d} &G_v|_{\calB_{n}(v)} \arrow{r}\arrow{d} &1\\
1\arrow{r} &U_v^n|_{\calB_{n+1}(v)}\arrow{r} &U_v|_{\calB_{n+1}(v)}\arrow{r} &U_v|_{\calB_{n}(v)} \arrow{r} &1
\end{tikzcd}
\]
 By the Five Lemma and the induction hypothesis, we will get the desired equality $G_v |_{B_{n+1}(v)} = U_v |_{B_{n+1}(v)}.$


 We dedicate the rest of this section to prove that 
 \begin{equation}\label{eq: the inductive step equation}
     G_v^n|_{\calB_{n+1}(v)}=U_v^n|_{\calB_{n+1}(v)}.
 \end{equation}

 
The proof of \eqref{eq: the inductive step equation} will follow similar ideas to that of \Cref{local action options}. 
Namely, we use simplicity of $\Alt_m$ ($m\geq 5$) to deduce that if $G^n_v|_{\calB_1(w)}$ is non-trivial for some partly free $w\in \calS_n(v)$ then it must be $\simeq \Alt_m$. 
After that, we will upgrade the above to hold for all  
partly free vertices in $\calS_n(v)$. 
Finally, we will analyse the components of $G_v^n|_{\calB_{n+1}(v)}$ to show that it is isomorphic to a direct product of $\Alt_m$ similar to the description of $U_v^n|_{\calB_{n+1}(v)}$ given in  \Cref{U_v^n is a product}.



\begin{claim}\label{non trivial implies alt}
    Let $p=\{v_0,\dots,v_{n-1}\}$ be a sequence of vertices such that $d(v_i,v_{i+1})\le 1$, and $x$ is a neighbour of $v_{n-1}$, then 
    $$G^1_p|_{\calB_1(x)}=1\text{ or }\simeq\Alt_m,$$
    In particular, if $x\in \calS_n(v)$ is partly free, then
    $$G^n_v|_{\calB_1(x)}=1\text{ or }\simeq\Alt_m.$$
\end{claim}
\begin{proof} 
     Let $S$ be the set of squares incident to both $v_{n-1}$ and $x$.
     Using \Cref{fact: fix normal in stab} we have: 
     $$ G^1_{\{v_i\}_{i=0}^{n-1}} \normalin G^1_{\{v_i\}_{i=1}^{n-1}} \normalin\ldots\normalin G^1_{v_{n-1}} \normalin \Fix_G(S).$$
     Restricting to $\calB_1(x)$ we obtain,
     $$ G^1_{\{v_i\}_{i=0}^{n-1}}|_{\calB_1(x)} \normalin G^1_{\{v_i\}_{i=1}^{n-1}}|_{\calB_1(x)} \normalin\ldots\normalin G^1_{v_{n-1}}|_{\calB_1(x)} \normalin \Fix_G(S)|_{\calB_1(x)}.$$
     By the assumption of the local action, $\bfD_x(G_x)=\Alt_{\ell}$, and by \Cref{fix in the Odd graph} the subgroup $\bfD_x(\Fix_G(S)) = \Alt(A_x)\simeq \Alt_m$ where $A_x$ is the label of the edge connecting $v_n$ to $x$.
     Equivalently, $\Fix_G(S)|_{\calB_1(x)} \simeq \Alt_m$.
     Since $m\geq5$, $\Alt_m$ is simple, and the subnormal subgroup $G^1_p|_{\calB_1(x)}$ is therefore trivial or the whole $\Alt_m$.

     For the second part of the lemma, consider  $x\in \calS_n(v)$ a  partly free vertex, and $[v\nea x]=\{v=v_0\ldots,v_{n-1},x\}$. Let $p=\{v_0,v_1,\dots,v_{n-1}\}$, then by \Cref{normal paths and pf}, $x$ is a neighbour of $v_{n-1}$.  So, by the first part of the lemma, $G^1_p|_{\calB_1(x)}=1$ or $\simeq \Alt_m$.
     Moreover, since $v\in p$, $G^n_v \normalin G^1_p$, so $G^n_v|_{\calB_1(x)} \normalin G^1_p|_{\calB_1(x)}$. By the simplicity of $\Alt_m$ it follows that $G^n_v|_{\calB_1(x)}=1$ or $\simeq\Alt_m$.
     \end{proof}


\begin{claim}\label{intersection of path fix} 
    $$G^n_v = \bigcap_{y'} G^1_{[v\nea y']}$$
    where $y'$ runs over the partly free vertices in $\calS_{n-1}(v)$.
\end{claim}
\begin{proof} 
First note that
$G^n_v 
\subseteq \bigcap_{y'} G^1_{[v\nea y']}$ is clear because $\calB_n(v)\supseteq \calN_1
[v\nea y']$ for every $y'\in \calS_{n-1}(v) $. 

For the other inclusion, 
let $g\in \bigcap_{y'} G^1_{[v\nea y']}$.
    We want to show, that in fact, $g$ fixes all of $\calB_n(v)$. 
    By \Cref{U_v^n is a product} it suffices to show that $g$ fixes $\calB_{n-1}(v)$ and the 1-neighbourhood $\calB_1(x)$ of each partly free vertex $x\in\calS_{n-1}(v)$. 
    It is easy to see that every vertex $z$ in $\calB_{n-1}(v)$ is in the 1-neighbourhood of some normal path $[v\nea y]$ for some $y\in \calS_{n-1}(v)$ and so $g$ fixes $z\in \calN_1([v\nea y])$. Clearly, for every partly free $x\in \calS_{n-1}(v)$, $g$ fixes the 1-neighbourhood $\calN_1([v\nea x])$ and in particular it fixes $\calB_1(x)$.
\end{proof}

     Denote by $[y'\nwa v \nea y]=[v\nea y']\cup [v\nea y]$.
\begin{claim}\label{Reduction to paths}
     Let $x\in \calS_n(v)$ be a partly free vertex, and let $y\in \calS_{n-1}(v)$ be the vertex connected to $x$ in $\calS_{n-1}(v)$.
    Then for every partly free $y'\in \calS_{n-1}(v)$, we have $G^1_{[y'\nwa v\nea y]}|_{\calB_1(x)}\ne 1$.
\end{claim}

\begin{proof} 

    Let $y'\in \calS_{n-1}(v)$ be a partly free vertex, and 
    let $\bar{v}$ be the first step of $[v\nea x]$, i.e., $\bar{v}=[v\nea x]\cap \calS_1(v)$, as shown in  \Cref{fig: reduction to paths}   
    \begin{figure}
        \centering
        \includegraphics{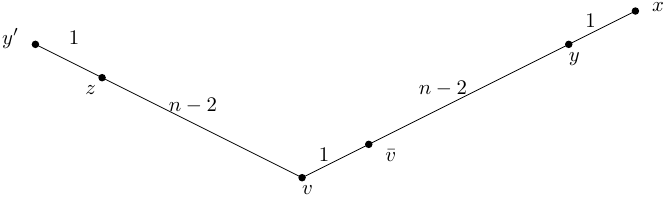}
        \caption{An example for normal path  $[y'\nwa v \nea x]$} \label{fig: reduction to paths}
    \end{figure} 

\bigskip
    
    \textbf{Case 1.}
    Assume that both  $d(\bar{v},y')\le n-2$ and $n\ge 3$. 
    
    Since $d(\bar{v},y)=n-2$, we have $y,y',v\in \calB_{n-2}(\bar{v})$. 
    Since $y'$ is partly free, the vertices in $\calN_1([y'\nwa v])$ are either in $\calB_1(y')$ or $\calB_{n-2}(v)$.
    In both cases, they are in $\calB_{n-1}(\bar v)$. 
    Clearly, $\calN_1([v\nea y])\subseteq \calB_{n-1}(\bar v)$.
    Thus, 
    $$\calN_1([y'\nwa v \nea y]) \subseteq \calB_{n-1}(\bar{v}) \quad \text{and}\quad  x\in \calB_{n-1}(\bar{v}).$$
    
    It follows  by the induction hypothesis that
    $$ G^1_{[y'\nwa v \nea y]}|_{\calB_1(x)} \ge G^{n-1}_{\bar{v}}|_{\calB_1(x)} = U^{n-1}_{\bar{v}}|_{\calB_1(x)} \ne 1.$$
    and so $G^1_{[y'\nwa v \nea y]}|_{\calB_1(x)}\ne 1$ as desired.

    \bigskip

    \textbf{Case 2.} 
    Assume either   $d(\bar{v},y')\ge n-1$, or $n=2$, 
    and assume for the sake of contradiction that $G^1_{[y'\nwa v\nea y]}|_{\calB_1(x)}= 1$.

    Let $z$ be the neighbour in $\calS_{n-2}(v)$ of the partly free vertex $y'$. 

    \textbf{Step 1.} $v\in \calB_{n-1}(\bar v)$ and if $v\in \calS_{n-1}(\bar v)$ then $v$ is not in the extended block of $x$.

    Clearly $d(\bar v,v)=1 \le n-1$, so $v\in \calB_{n-1}(\bar v)$. If $v\in \calS_{n-1}(\bar{v})$, then $n=2$. In this case, we are considering $\calS_1(\bar{v})$, where all blocks are of size one. So for $v$ to be in the extended block of $x$, we need that $d(x,v)=1$, but $x\in \calS_2(v)$. 
    
    \textbf{Step 2.} $z\in \calB_{n-1}(\bar v)$ and if $z\in \calS_{n-1}(\bar v)$ then $z$ is not in the extended block of $x$.

    We have $d(z,\bar v)\le d(z,v)+d(v,\bar v) = n-1$ so $z\in \calB_{n-1}(\bar v)$.
    Assume $z\in \calS_{n-1}(\bar v)$. 
    Consider the hyperplane $\calH$  separating $x$ and $y$. Let $h$ be the halfspace containing $x$ and $h^*$ the complementary halfspace.
    The extended block of $x$ is contained in $h$. It suffices to show that $z\in h^*$.
    Indeed, since $n=d(v,x)$, it follows that $d(v,\calH)\ge n-1$, but $d(v,z)= n-2$. Hence, $z$ is on the same side of $\calH$ as $v$, i.e., $z\in h^*$.
    
    

    \textbf{Step 3.} 
    $ G^1_{[z\nwa v\nea y]}|_{\calB_1(x)} \simeq \Alt_m$.
    
    Let $\sigma\in \Alt_m$.
    By the previous steps, the vertices $z,v$ are in $\calB_{n-1}(\bar v)$ and are not in the extended block of $x$ in $\calS_{n-1}(\bar{v})$. 
    By \Cref{U_v^n is a product}, we claim that there exists an element $g\in G^{n-1}_{\bar{v}}$ such that 
    $\bfD_x g =\sigma$ but $\bfD_v g=\bfD_z g=1$: This is immediate if each of $x,z$ is either partly free in $\calS_{n-1}(\bar v)$ or in $\calB_{n-2}(\bar v)$. If either of $z,v$ is a free vertex in $\calS_{n-1}(\bar{v})$, then by part (3) of \Cref{U_v^n is a product}, the local action of $g$ at that vertex is determined by the blocks it is connected to by an edge, neither of which is the block containing $x$.
    This shows that $\bfD_xg$ can be chosen independently of $\bfD_zg$ and $\bfD_vg$.
    
    Finally, since $\calN_1([z\nwa v \nea y]) \subseteq \calB_{n-1}(\bar{v})\cup \calB_1(v)\cup \calB_1(z)$, we get that $g\in G^1_{[z\nwa v \nea y]} |_{\calB_1(x)}.$
    And so $G^1_{[z\nwa v \nea y]} |_{\calB_1(x)}\simeq \Alt_m$.
    This finishes the proof of Step 3.

    \bigskip
    
    By \Cref{normal paths and pf} $z$ is the penultimate vertex in $[v\nea y']$. Hence $[y'\nwa z]\cup[z\nwa v\nea y]= [y'\nwa v\nea y]$ and, 
    \begin{equation*}\label{intersection of stabs}
        G^1_{y'}\cap G^1_{[z\nwa v\nea y]} =G^1_{[y'\nwa v\nea y]}
    \end{equation*}
    Using our assumption that $G^1_{[y'\nwa v\nea y]}|_{\calB_1(x)}=1$, we get that for every $g\in G^1_{[z\nwa v\nea y]}$:
    \begin{equation}\label{eq: fix one ball fix the other}
    g|_{\calB_1(y')}=1 \implies g|_{\calB_1(x)}=1.
    \end{equation}

    We consider the restrictions of $G^1_{[z\nwa v\nea y]}$ to the balls of radius 1 around $y',x$:
    $$P_{y'}:G^1_{[z\nwa v\nea y]} \to  G^1_{[z\nwa v\nea y]}|_{\calB_1(y')} \text{ and } P_{x}:G^1_{[z\nwa v\nea y]}\to G^1_{[z\nwa v\nea y]}|_{\calB_1(x)}.$$
  
    We saw in \eqref{eq: fix one ball fix the other} that $\ker P_{y'} \le \ker P_x$, therefore $P_x$ factors through $P_{y'}$. That is, there is a homomorphism $\iota: \im(P_{y'})\to \im(P_x)$ such that $P_x=\iota \circ P_{y'}$.

    \textbf{Step 4.} The map $\iota : G^1_{[z\nwa v\nea y]}|_{\calB_1(y')}\to  G^1_{[z\nwa v\nea y]}|_{\calB_1(x)}$ is an isomorphism.

    Note that $\iota$ is surjective by definition. 
    We will  show that $\iota$ is an isomorphism by showing that
    $|\im(P_x)|\ge |\im(P_{y'})|$:



    Any element in $\im(P_{y'})$ fixes $\calB_1(z)$, and in the link $\Lk(y',X_L)$, it fixes the star of the vertex corresponding to the edge that connects $y'$ to $z$. By \Cref{fix in the Odd graph}, such a fixator is isomorphic to $\Alt_m$. Therefore, $|\im(P_{y'})| \le |\Alt_m|$.
    And by the previous step, $|\im(P_x)| =|\Alt_m|$.
    This shows that $|\im(P_{y'})|\le |\im(P_x)|$ and hence that $\iota$ is an isomorphism. 
    
    \textbf{Step 5.} There exists $g\in G_v^{n-2}$ such that $g(x)=x$, $g(z)=z$ and $g(y')\ne y'$.

    Let $e$ be an edge incident to $z$, such that the hyperplane $\calH$ through $e$ separates $z$ from $y$ and $v$ (unless $n=2$, in which $\calH$ separates $z=v$ from $y$).
    Let $h^*$ be the halfspace containing $y$. Then we have $z,v,y,x\in\calN_1(h^*)$.
    
    Let $Y_e$ be the label of the hyperplane $\calH$. Let $Y'$ be the label of the edge corresponding to $y'$ in $\Lk(z,X)$. Let $\phi \in \Alt(Y_e)$ be such that $\phi(Y')\ne Y'$.
    As in the proof of \Cref{U_v^n is a product}, there exists an element $\Phi'\in U$ such that $\bfD_z(\Phi')=\phi$ and $\Phi'$ is the identity on $\calN_1(h^*)$.
    In particular $\Phi'(v)=v, \;\Phi'(x)=x, \;\Phi'(z)=z$ and $ \Phi'(y')\ne y'$

    By the induction hypothesis there exists $g\in G_v$ such that $g|_{\calB_n(v)} = \Phi'|_{\calB_n(v)}$. So, $g(v)=v, \;g(x)=x,\; g(z)=z$ and $g(y') \ne y'$.

    \textbf{Step 6.} Deriving a contradiction.

    Let $g$ as in Step 5. 
    Denote $y''=g(y')$.

    Since $g\; G^1_{[y'\nwa v\nea y]}\;g\ii= G^1_{[y''\nwa v\nea y]}$ and $g(x)=x$ we get 
    $G^1_{[y''\nwa v\nea y]}|_{\calB_1(x)}= 1$
    Then, as for $y'$, we can consider  the restriction 
    $P_{y''}:G^1_{[z\nwa v\nea y]} \to  G^1_{[z\nwa v\nea y]}|_{\calB_1(y'')}$. 
    By Step 4 the map $\iota':\im(P_{y''})\to \im(P_x)$ will also be a bijection. 
    The composition of $(\iota')\ii\circ\iota: \im(P_{y'})\to \im(P_{y''})$ is injective. 
    It follows that  $$G^1_{[z\nwa v\nea y]}\cap G^1_{y'} \le  G^1_{[z\nwa v\nea y]}\cap G^1_{y''}$$
    Since $G^{n-1}_v\cap G^1_y\le G^1_{[z\nwa v\nea y]}$, we get 
    \begin{equation}\label{eq: intersection}
    G^{n-1}_v\cap G^1_y\cap G^1_{y'} \le G^{n-1}_v\cap G^1_y\cap G^1_{y''}
    \end{equation}

    On the other hand, $y',y''$ are partly free in $\calS_{n-1}(v)$ in different blocks, and $y$ is not in the extended block of $y',y''$. So by \Cref{U_v^n is a product} there is an element $h\in G^{n-1}_v$ such that $h|_{\calB_1(y)}=\id$, $h|_{\calB_1(y')}=\id$ but $h|_{\calB_1(y'')}\ne \id$.
    This contradicts \eqref{eq: intersection}.
    \end{proof}

\begin{claim}\label{local action in pf is alt}
    If $x\in\calS_n(v)$ is partly free, 
    then $G^n_v|_{\calB_1(x)}\simeq \Alt_m$.
    In particular, $\Fix_X(G^n_v)=\calB_n(v)$.
\end{claim}

\begin{proof}
    
    For $n\ge 2$ Let $x\in \calS_{n}(v)$ be partly free.
    By \Cref{intersection of path fix}, 
    $G^n_v = \bigcap_{y'} G^1_{[v \nea y']}$ where $y'$ runs over the partly free vertices in $\calS_{n-1}(v)$.
    In particular, for every $y\in \calS_{n-1}(v)$, we have: 
    $$G^n_v = \bigcap_{y'} G^1_{[y'\nwa v \nea y]}$$
    Restricting to $\calB_1(x)$ we get
    \begin{equation}\label{intersecting the path stabs}
        G^n_v|_{\calB_1(x)} = \bigcap_{y'} G^1_{[y'\nwa v \nea y]}|_{\calB_1(x)}.
    \end{equation}
    Both $G^n_v|_{\calB_1(x)}$ and $G^1_{[y'\nwa v \nea y]}|_{\calB_1(x)}$ fix the star of a vertex in $\Lk(x,X)$ and so they are both subgroups of a group isomorphic to $\Alt_m$.
    By \Cref{Reduction to paths} we have that $G^1_{[y'\nwa v \nea y]}|_{\calB_1(x)}\ne 1$, and  then \Cref{non trivial implies alt} implies $G^1_{[y'\nwa v \nea y]}|_{\calB_1(x)}\simeq\Alt_m$.  By \eqref{intersecting the path stabs}
     the result follows.   
\end{proof}

Let us recall that for any vertex $a\in \calS_1(v)$, we defined the \textit{$n$-th sector of $a$} to be $\calS_n(v,a):=\{x\in P_n(v)\ |\  a\in [v\sea x]  \}$, where $P_n(v)$ is the set of partly free vertices in $\calS_n(v)$.

    
\begin{claim}\label{sector is universal}
    For every $a\in \calS_1(v)$,
    $$G^n_v|_{\calN_1(\calS_n(v,a))}=U^n_v|_{\calN_1(\calS_n(v,a))}.$$
\end{claim}
\begin{proof}
    We have $ G_{v}^n\triangleleft G_a^{n-1}$, therefore (because $\calS_n(v,a)$ is not empty) 
    $$G_{v}^n|_{\calN_1(\calS_n(v,a))} \triangleleft G_a^{n-1}|_{\calN_1(\calS_n(v,a))}= U^{n-1}_a |_{\calN_1(\calS_n(v,a))}\simeq \prod_b \Alt_m.$$  Where $b$ runs over all 
    blocks in $\calS_n(v,a)$.
    Now, since $ G_{v}^n|_{\calN(\calS_n(v,a))}$ is normal inside a direct product of simple groups $\prod_b \Alt_m$, $G_{v}^n|_{\calN(\calS_n(v,a))}$ is itself a product  of some of the factors of $\prod_b \Alt_m$.
    In addition, every block $b$ contains a vertex  $x\in \calS_n(v,a)$, and from \Cref{local action in pf is alt} we have $G_{v}^n|_{\calN_1(\calS_n(v,a))}|{_{\calB_1(x)}}=\Alt_m$, 
    therefore $G_{v}^n|_{\calN(\calS_n(v,a))}=\prod_b \Alt_m=U^n_v|_{\calN_1(\calS_n(v,a))}$. 
\end{proof}



\begin{claim}
    $$G^n_v|_{\calB_{n+1}(v)}=U^n_v|_{\calB_{n+1}(v)}.$$ 
\end{claim}

\begin{proof} 
To reduce notation, let us denote $\bbG = G^n_v|_{\calB_{n+1}}$ and $\bbU = U^n_v|_{\calB_{n+1}(v)}$.
We want to show $\bbG = \bbU$.

By assumption we know $\bbG \le \bbU$. 
By \Cref{U_v^n is a product}, $\bbU$ is a product of $\prod_{b\in Bl_n(v)}\Alt_m$. Let us denote by $B = Bl_n(v)$, and for each $b\in B$ let $S_b\le \bbU$ be the factor $\Alt_m$ corresponding to the block $b$. Let $\pi_b:\bbU \to S_b$ be the projection.

For a component $C\le \bbG$, by \Cref{properties of components}.2, $\pi_b(C) = 1 \text{ or } S_b$. 
Let $B_C = \{b\in B \;|\; \pi_b(C)\simeq S_b\}$.
    If $C\ne C'$ are distinct components then by \Cref{properties of components}.3 $[C,C']=1$, $[\pi_b(C),\pi_b(C')]=1$ and so at most one of them can be non-trivial. Therefore $B_C\cap B_{C'}=\emptyset$.
Note that $B_C=\{b\}$ if and only if $C = S_b$.

We want to show that $\bbG = \bbU = \prod_{b\in B}S_b$.
It suffices to show that $\bbG$ is a product $\prod_{b\in B'} S_b$ for some $B'\subseteq B$, since by \Cref{local action in pf is alt}, $\pi_b(\bbG)\ne 1$ for all $b\in B$, and so we must have $B' = B$, and $\bbG = \bbU$.

Consider the sequence of subgroups of $\bbG$ defined inductively by: 
\begin{itemize}
    \item $\bbG_0=1$, and  
    \item $\bbG_{i+1}/\bbG_i$ is the subgroup of $\bbG/ \bbG_i$ generated by the components of $\bbG/\bbG_i$.
\end{itemize}
Then, $\bbG_i \normalin \bbG$ is characteristic and the series terminates in finitely many steps:
$$1=\bbG_0\normalin \bbG_1 \normalin \dots \normalin \bbG_r= \bbG$$
The quotient $\bbG_{i+1}/\bbG_{i}$ is the direct product of the components of $\bbG/\bbG_i$.

By induction on $i$ we will prove that there exists a subset $B_i \subseteq B$ such that $\bbG_i = \prod_{b\in B_i}S_b$. In particular, $\bbG = \bbG_r$ is such a product.
\footnote{\emph{A posteriori}, we must have $r=1$.}

    The base case $i=0$ holds trivially for $B_0=\emptyset$.

    Assume $\bbG_i = \prod_{b\in B_i}S_i$. 
    Note that $\bbG_i$ is normal in $\bbG$ and $\bbU$. Denoting $\bar\bbG = \bbG/\bbG_i$ and $\bar \bbU = \bbU/\bbG_i$ we have  $$\bar \bbG \le \bar\bbU \simeq \prod_{b\in B-B_i}S_b$$
    
    For each component $\bar C\le \bar \bbG$, define $\bar B_{\bar C}=\{b\in B-B_i\;|\;\pi_b(\bar C)= S_b\}.$
    Then, as above, if $\bar C\ne \bar C'$ are distinct components of $\bar \bbG$ then $\bar B_{\bar C}\cap \bar B_{\bar C'}=\emptyset$.
    Let $$B_{i+1} := B_i \sqcup \bigsqcup_{\bar C} \bar B_{\bar C}$$ where the union is taken over all components $\bar C\le \bar \bbG$.

    To show that $\bbG_{i+1} = \prod _{b\in B_{i+1}}S_b$ it suffices to prove that $\bbG_{i+1}/\bbG_i = \prod _{b\in B_{i+1}-B_i}S_b$. For which it suffices to prove that $\bar B_{\bar C}$ is a singleton for every component $\bar C \le \bar \bbG$.

    Let $a$ be a neighbour of $v$. By \Cref{sector is universal}, we have $$\bbG|_{\calN_1(\calS_n(v,a))} = \bbU|_{\calN_1(\calS_n(v,a))} = \prod_{b\in B\;:\;b\subseteq \calS_n(v,a)} S_b.$$ Further restricting to those blocks which are not in $B_i$ we get
    $$\bar \bbG|_{\calN_1(\calS_n(v,a))} = \bar \bbU|_{\calN_1(\calS_n(v,a))} = \prod_{b\in B-B_i\;:\;b\subseteq \calS_n(v,a)} S_b$$
    If $\bar C \le \bar \bbG$ is a component, then under the restriction above it is either trivial or $S_b$ of some block $ b\in B-B_i$ in the sector $\calS_n(v,a)$.
    This shows that $\bar B_{\bar C}$ can have at most one block per sector, i.e., 
    \begin{equation}\label{components in sectors}
        |\bar B_{\bar C} \cap \{b\in B - B_i\;|\;b\subseteq \calS_n(v,a)\}|\le 1
    \end{equation}

    Assume for the sake of contradiction that $b_1,b_2\in \bar B_{\bar C}$ are distinct blocks. Then, by the above, $b_1,b_2$ belong to distinct sectors. By 
    \Cref{auxiliary claim for components in universal}, there exists an element $g'\in U^{n-1}_v$ such that, without loss of generality, $g'(b_1)=b_1$ and $g'(b_2)\ne b_2$. 
    By the induction hypothesis, there exists an element $g\in G_v$ such that $g|_{\calB_n(v)} = g'|_{\calB_n(v)}$ and so $g(b_1)=b_1$ and $g(b_2)\ne b_2$ as well.
    Since $g\in G^{n-1}_v$ (and $n\ge 2$), the blocks $g(b_2),b_2$ are in the same sector. By \eqref{components in sectors}, we thus have $g(b_2)\notin \bar B_{\bar C}$. 
    
    Finally, $G_v$ acts on $\bar\bbG$ by conjugation. The conjugate $\bar C'=g \bar C g\ii$ is a component of $\bar \bbG$, and $\bar B_{\bar C'} = g\bar B_{\bar C}.$ 
    Since $\bar B_{\bar C}\ni b_1\in \bar B_{\bar C'}$ but $\bar B_{\bar C}\not\ni g(b_2)\in \bar B_{\bar C'}$ the sets $\bar B_{\bar C},\bar B_{\bar C'}$ are neither disjoint nor equal. This is a contradiction.
 \end{proof}

From the induction hypothesis, and the Five lemma, 
by the last claim we get the desired equality
    $$G_v|_{\calB_{n+1}(v)}=U_v|_{\calB_{n+1}(v)}.$$
    This finishes the proof of \Cref{characterization of dense subgroups}.
\end{proof}


\section{Constructing lattices in products}\label{construction}

Recall the following definition:

\begin{definition}
A \emph{BMW group of degree $(m,n)$} is a group $\Gamma\le \Aut(T_m)\times \Aut(T_n)$ that acts simply transitively on the vertex set of the product $T_m\times T_n$ of the $m$-regular tree and the $n$-regular tree.
A BMW group is \emph{involutive} if it is generated by involutions that reflect along edges incident to a fixed vertex.
\end{definition}

For further background on such BMW groups, we refer the reader to \cite[Section 4]{caprace2019finite} and \cite{lazarovich2022counting}. 

In this section we prove the following: 

\begin{theorem}\label{building the lattice}
    For every involutive BMW group $\Gamma$ of degree $(m,n)$ with alternating local actions, there exist $c\ge m$, $d\ge n$ and a lattice $\Lambda \le \Aut(T_{c}) \times \Aut(X_{O_d})$ such that:
    \begin{itemize}
        \item The group $\Gamma$ is a subgroup of a hyperplane stabilizer of $\Lambda$. In particular, if $\Gamma$ is irreducible then so is $\Lambda$.
        \item $\Lambda$ acts on $T_{c} \times X_{O_d}$ simply transitively on vertices.
        \item The local actions of $\Lambda$ are alternating.
    \end{itemize}
\end{theorem}

Here is an outline of this section:
In \S\ref{uniqueness of regular}, using the uniqueness theorem of \cite{lazarovich2018regular}, we will show that $T_c \times X_{O_d}$ is the unique CAT(0) cube complex all of whose links are isomorphic to the join of a set of $c$ points and the graph $O_d$. Using this, in \S\ref{building a lattice in the product}, we will describe a combinatorial structure, called an ``interlacing pair'', which defines a group acting simply transitively on the vertices of $T_c \times X_{O_d}$.
In \S\ref{lattice containing a BMW}, given an involutive BMW group $\Gamma$ as in the theorem, we will construct a specific interlacing pair that ``contains'' the defining data of $\Gamma$. 
Finally, in \S\ref{proof of building the lattice}, we will finish the proof of \Cref{building the lattice} by showing that the group corresponding to the interlacing pair constructed in \S\ref{lattice containing a BMW} has the desired properties.


\subsection{Uniqueness of regular CAT(0) cube complexes}
\label{uniqueness of regular}

Let us recall the definition for a superstar-transitive simplicial complex, and the uniqueness result of \cite{lazarovich2018regular}. 

\begin{definition} \label{def superstar}
A simplicial complex $L$ is \textit{superstar-transitive} if for any two simplicies $\sigma,\sigma'$, and an isomorphism $\phi:\calN_1(\sigma)\to \calN_1(\sigma')$ with $\phi(\sigma)=\sigma'$, there exists an automorphism $\Phi:L\to L$ such that $\Phi|_{\calN_1(\sigma)}=\phi$.
\end{definition}

\begin{theorem}\cite[Theorem 1.2]{lazarovich2018regular}\label{uniquness of CAT(0)}
    Let $L$ be a finite flag simplicial complex. 
    The associated Davis complex $X_{L}$ is the unique CAT(0) cube complex whose vertex links are all isomorphic to $L$
    if and only if $L$ is superstar-transitive.
\end{theorem}

\begin{proposition} \cite[Proposition 5.3]{lazarovich2018regular}\label{O_d is sst}
    For every $d\in \bbN$, the complex $L=O_d$ is superstar-transitive. 
\end{proposition}

The link  at every vertex in $T_c\times  X_{O_d}$ is the join $L'=Z*L$ 
of a set $Z$ of $c$ isolated vertices with the Odd graph $L=O_d$.
We recall that the join $L_1*L_2$ of two simplicial complexes $L_1,L_2$ is the complex whose vertex set is $V(L_1) \sqcup V(L_2)$ and all simplices of the form $\sigma_1\cup \sigma_2$ where where $\sigma_1,\sigma_2$ are a simplices in $L_1,L_2$ respectively. 
In the next proposition we show that $L'$ is superstar-transitive. 

\begin{proposition}\label{the join is superstar-transitive}
    Let $3\leq d\in \bbN$, 
    Let $Z=\{z_1,\ldots,z_c\}$ be the complex of $c$ isolated vertices, 
    and let $L=O_d$ be the Odd graph with parameter $d$.
    Then the join $L'=Z*L$ 
    is superstar-transitive
\end{proposition}

The following observation will be useful in the proof of the proposition.

\begin{lemma}\label{reduction of superstar}
    To verify that $L$ is superstar-transitive it suffices to be able to extend an isomorphism $\phi:\calN_1(\sigma) \to \calN_1(\sigma')$ to $\Aut(L)$ (as in \Cref{def superstar}) in two cases:
    \begin{enumerate}[label = (\arabic*)]
        \item $\sigma,\sigma'$ are vertices, and
        \item $\sigma=\sigma'$ are $n$-simplices and $\phi|_\sigma=\id$.
    \end{enumerate}
\end{lemma}
\begin{proof}
    Assume (1) and (2) hold. Let $\sigma,\sigma', \phi$ be as in the assumptions of \Cref{def superstar}, let $v\in V(\sigma)$, and let $v'\in V(\sigma')$ be the image of $v$ under $\phi$. Define $\psi:\calN_1(v)\to \calN_1(v')$ to be the restriction of $\phi$ to $\calN_1(v)$, note that in particular, $\psi$ is defined on $\sigma$, and $\psi(\sigma)=\sigma'$. 
    By (1), there is an automorphism $\Psi:L\to L$, such that $\Psi|_{\calN_1(v)}=\psi$, in particular $\Psi(\sigma)=\sigma'$, and so $\Psi(\calN_1(\sigma))=\calN_1(\sigma')$. Now consider $\varphi:= \Psi^{-1}|_{\calN_1(\sigma')}\circ\phi: \calN_1(\sigma)\to \calN_1(\sigma)$, note that $\varphi|_{\sigma}=\id$. 
    By (2), there exists an automorphism $\Phi':L\to L$ such that $\Phi'|_{\calN_1(\sigma)}=\varphi$.  Take $\Phi:=\Psi \circ \Phi'$, this is clearly an automorphism $L\to L$, moreover $$\Phi|_{\calN_1(\sigma)}= \Psi|_{\calN_1(\sigma)} \circ \Phi'|_{\calN_1(\sigma)}= \Psi|_{\calN_1(\sigma)} \circ \varphi= \Psi|_{\calN_1(\sigma)} \circ \Psi^{-1}|_{\calN_1(\sigma')}\circ\phi= \phi. \qedhere$$ 
    \end{proof}

\begin{proof}[Proof of \Cref{the join is superstar-transitive}]
    By \Cref{reduction of superstar} it suffices to check \Cref{def superstar} in two cases (1) and (2) of the lemma.

    \textbf{(1)} Let us show that the condition for superstar-transitivity holds when $\sigma,\sigma'$ are vertices. Let $v,v'\in V(L')$ be two vertices. 
    There are two cases to consider $v\in Z$ and $v\in L$.
    By \Cref{O_d is sst}, $L$ is superstar-transitive, and clearly so is $Z$.
    
    If $v\in Z$, then $\calN_1(v,L')={\calN_1(v,Z)}*L=\{v\}*L$, this implies that also $v'\in Z$. Let $\phi:\calN_1(v) \to \calN_1(v')$ be an isomorphism with $\phi(v)=v'$. In particular $\phi$ induces an automorphism on $L$. Take for example  $\Phi:Z*L\to Z*L$ such that $\Phi|_{L}=\phi|_L$, 
    $\Phi(v)=v'$, $\Phi(v)=v'$ and for any other $v,v'\ne w\in V(Z)$, $\Phi(w)=w$.

    If $v\in L$, then $\calN_1(v)=Z*{\calN_1}(v,L)=\{v\}*Z*D$ where $D=\{a_1,\ldots a_d\}$ a set of $d$ isolated vertices, the neighbours of $v$ in $L$. We assume $d\ne m$, therefore $v'\in L$ as well. Let $\phi:\calN_1(v) \to \calN_1(v')$ be an isomorphism with $\phi(v)=v'$, $\phi$ induces an isomorphism on $Z$, and also, since $L$ is superstar-transitive, and $\phi|_{L}:{\calN_1}(v,L)\to {\calN_1}(v',L)$ is an isomorphism that sends $v$  to $v'$, there is an automorphism $\Phi':L\to L$ such that $\Phi'|_{{\calN_1}(v,L)}=\phi|_L$. Take $\Phi$ to be $\phi$ on $Z$ and $\Phi'$ on $L$. 

    \textbf{(2)} Let us show that the condition for superstar-transitivity holds for a general simplex $\sigma$ in $Z*L$, and an isomorphism $\phi:\calN_1(\sigma)\to \calN_1(\sigma)$ such that $\phi|_{\sigma}=\id$. We can assume that $\sigma$ is not a vertex. A general simplex in a join, is a union of two simplices, one from each factor, where at most one of them might be the empty set. 
    If $\sigma\in L$, then by superstar-transitivity of $L$ the same process as before will work.   Assume $\sigma=v\cup\tau$, where $v$ is in $Z$, and $\tau$ is in $L$. Then, 
    $$\calN_1(v\cup\tau)= \calN_1(v\cup\{ \emptyset\} )\cup \calN_1(\{\emptyset\} \cup\tau)= {\calN_1}(v,Z)* L \cup_{{\calN_1(v,Z)}*{\calN_1}(\tau,L)} Z*{\calN_1}(
    \tau,L). $$
    This implies that $\phi|_Z$ and $\phi|_L$ are automorphisms on $Z,L$,  that agrees on ${\calN_1(v,Z)}*{\calN_1}(\tau,L)$, therefore we take $\Phi$ to be $\phi|_Z * \phi |_L$.  
\end{proof}

\subsection{Building a lattice in $T\times X_L$}\label{building a lattice in the product}
Our goal is to describe the combinatorial data needed in order to construct a lattice in $T_c\times X_L$, where $L=O_d$. Fix a set $Z$ of size $c$, and think of $Z$ as the graph with $c$ isolated vertices. Identify $T_c$ with the Davis complex $X_Z$.

\begin{definition}
    Consider a pair $\calD = (\{\zeta_z\}_{z\in Z},\{\delta_D\}_{D\in V(L)})$ consisting of two collections of involutions:
    \begin{enumerate}[label = (D\arabic*)]
    \item \label{zetas are involutions} $\zeta_z \in \Sym_\ell=\Aut(L)$ satisfying $\zeta_z^2=1$, and
    \item \label{deltas are involutions} $\delta_D\in \Sym(Z)$ for every $D\in V(L)$ satisfying $\delta_D^2=1$.
    \end{enumerate}
    The pair $\calD$ is \emph{interlacing} if moreover:
    \begin{enumerate}[label = (D\arabic*)]
    \setcounter{enumi}{2}
    \item \label{interlacing deltaD=1 or deltaD'=1}If $\{D,D'\} \in E(L)$, then $\delta_D=1$ or $\delta_{D'}=1$
    \item \label{interlacing condition on zetas} If $\delta_D(z)=z'$ then $\zeta_z(D)=\zeta_{z'}(D)$ and $\zeta_{z}|_{[\ell]-D} = \zeta_{z'}|_{[\ell]-D}$
    \item \label{interlacing condition on deltas} If $\zeta_{z}(D)=D'$ then $\delta_D (z) = \delta_{D'}(z)$
\end{enumerate}
\end{definition}


\begin{proposition}\label{the lattice Lambda}
    Given an interlacing pair $\calD = (\{\zeta_z\}_{z\in Z},\{\delta_D\}_{D\in V(L)})$, the group given by the following presentation
    \begin{align*}\Lambda = \Big\langle Z \cup V(L)\quad  \big|\quad   &z^2\quad \forall \; z\in Z, \\
    &D^2\quad \forall D\in V(L), \\ 
    &[D,D']\quad  \forall \{D,D'\}\in E(L),\\
    &z \cdot D \cdot \delta_D(z)\cdot \zeta_z(D) \quad \forall z\in Z, \forall D\in V(L)\Big\rangle
    \end{align*}
is a vertex transitive lattice in $\Aut(X_Z)\times \Aut(X_L)$ where $X_Z$ is the $s$-regular tree and $X_L$ is the Davis complex with defining graph $L=O_d$.  Moreover, the local actions on $X_Z$ and $X_L$ are  $\gen{\{\delta_D \}_{D\in V(L)}}\le \Sym(Z) $  and $\gen{\{\zeta_z\}_{z\in Z}}\le \Aut(L)$ respectively.
\end{proposition}

\begin{proof}
    
    Consider the presentation complex $K$ for $\Lambda$. Let $\tild K$ be its universal cover. Consider the complex $Y$ 
    obtained from $\tild K$ by performing the following  operations:
    \begin{enumerate}
        \item Collapse each disk of the form $z^2$ or $D^2$ to an edge. This eliminates all bigons, and now there is a unique edge labeled $z$ or $D$ incident to each vertex.
        \item Every relation of the form $z \cdot D \cdot \delta_D(z)\cdot \zeta_z(D)$ corresponds to possibly more than one disk sharing the same boundary. Identify these disks to one disk so that every 4-cycle is the boundary of exactly one disk.  
        \item Fill every 2-skeleton of a 3-cube, 
        with a 3-cube.
    \end{enumerate}
    It is not hard to see that the complex $Y$ is simply connected, and that $\Lambda$ acts on $Y$ simply transitively on vertices.

    Another way of looking at $Y$ is to consider the unoriented Cayley graph of $\Lambda$ (all generators are of order two, so instead of two oriented edges we put one unoriented edge). Then insert a disk for any relation in the presentation, and continue with filling any 2-skeleton of a 3-cube, with a 3-cube.

    To show that $Y \simeq X_Z \times X_L$, note that $X_Z \times X_L$ is isomorphic to the Davis complex $X_L'$ whose defining graph is the join $L'=Z*L$. By \Cref{the join is superstar-transitive}, $L'$ is superstar-transitive, and so by \Cref{uniquness of CAT(0)} it suffices to show the following claim. 
    
    \begin{claim}\label{the complex is the product}
        The link of each vertex in $Y$ is isomorphic to the join $L'=Z*L$
    \end{claim}
    
    \begin{proof}[Proof of \Cref{the complex is the product}] 
    Consider a vertex $v\in Y$.

    \textbf{Vertices:} The edges incident to $v$ are in bijection with the generators of $\Lambda$. This shows that the vertices in the link $\Lk(v,Y)$ correspond exactly to  the set $Z\cup V(L)$ of generators, which are also the vertices of $L'$. 
    
    \textbf{Edges:} An edge $e$ in $L'$ is  one of two forms: either $e=\{z,D\}$ for $z\in Z$ and $D\in V(L)$ or $e=\{D,D'\}\in E(L)$.

    If $e=\{z,D\}$, then the relation  $z \cdot D \cdot \delta_D(z)\cdot \zeta_z(D)$ in the presentation of $\Lambda$ gives rise to the unique square in $Y$ incident to the two edges in $Y$ that correspond to $z,D$. Therefore, there is a corresponding  edge in the link $\Lk(v,Y)$. 

    Similarly, if $e=\{D,D'\} \in E(L)$ then the relation $[D,D']$ gives rise to a square in $Y$ which shows that $\{D,D'\}$ is an edge in $\Lk(v,Y)$. Conversely, if $\{D,D'\}$ is an edge of $\Lk(v,Y)$ then there is a relation of the presentation of $\Lambda$ containing $DD'$ as a sub-word. The only relation of this form is $[D,D']$, which implies that $\{D,D'\}\in E(L)$.
    

    \textbf{2-simplices:} We are left to show that the 2-simplices in the link $\Lk(v,Y)$ are exactly the 2-simplices in $L'$. 
    Since $L'$ is a flag simplicial complex, it will suffice to show that so is $\Lk(v,Y)$.
    Let $\{z,D,D'\}$ be a 3-cycle in the link $\Lk(v,Y)$. We will show that it corresponds to a 3-cube in $Y$, and therefore it is the boundary of a 2-simplex in the link. 
    This 3-cube is made up of 6 square faces, we write down their labels below:

    
    Since $\calD$ is interlacing, by \ref{interlacing deltaD=1 or deltaD'=1} $\delta_D=1$ or $\delta_{D'}=1$. Without loss of generality, assume $\delta_{D'}=1$.  Denote ${\delta_D(z)=z'},\; \zeta_z(D)=\tilde{D},\; \zeta_z(D')=\tilde{D}'$.
    
    By the definition of $\Lambda$ all the following 4-tuples are the labels of squares in $Y$: 
    \begin{align}
        &\{D,D',D,D'\}, \label{first square} \\
        &\{ z,D,\delta_D(z),\zeta_z(D)\}= \{ z,D,z',\tilde{D}\}, \\
        &\{z,D',\delta_{D'}(z),\zeta_z(D') \}=\{ z,D',z,\tilde{D}' \}.
    \end{align}
    
    By \ref{zetas are involutions}, $\zeta_z\in \Sym_\ell=\Aut(L)$. Since $\{D,D'\}\in E(L)$ is an edge, also its image $\{\tilde{D},\tilde{D}'\}$ under the automorphism $\zeta_z$ is an edge of $L$. So the following 4-tuple is also the label of a square in $Y$:
    \begin{align}
    \{\tilde{D}, \tilde{D}',\tilde{D},\tilde{D}'\}
    \end{align}

    By \ref{interlacing condition on zetas}, $\zeta_z(D)=\zeta_{z'}(D)=\tilde{D}$. Moreover, since $\{D,D'\}\in E(L)$ --  or equivalently $D'\subseteq [\ell]-D$ -- we also have $\zeta_z(D')=\zeta_{z'}(D')=\tilde{D}'$. Therefore, the following 4-tuple is also the label of a square in $Y$: 
    \begin{align}
    \{ z',D',\delta_D'(z'),\zeta_{z'}(D')\}=\{ z',D',z',\tilde{D}'\}.
    \end{align}
    
    Finally, the following 4-tuple is the label of a square in $Y$: 
    \begin{align}
        \{ \delta_{D'}(z),D,\delta_D(\delta_{D'}(z)),\zeta_{\delta_{D'}(z)}(D)\}= {\{z,D, z' ,\tilde{D}\}}. \label{last square}
    \end{align}

    There are six 4-tuples in \eqref{first square}-\eqref{last square} form the the 2-skeleton of a 3-cube in $Y$, as shown in \Cref{fig: the cube}.

    \begin{figure}
        \centering
        \includegraphics[]{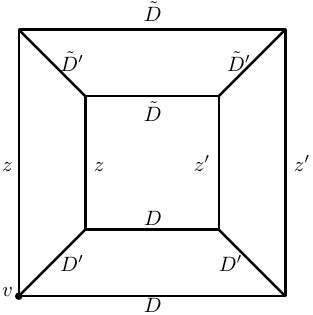}
        \caption{The cube which corresponds to the 2-simplex $\{z_i,D,D'\}$ in $\Lk(v,Y)$}\label{fig: the cube}
    \end{figure} 
    \end{proof}

    To finish the proof of the proposition, we have to show that the local actions on  $X_Z$ and $X_L$ are  $\gen{\{\delta_D \}_{D\in V(L)}}\le \Sym(Z) $  and $\gen{\{\zeta_z\}_{z\in Z}}\le \Aut(L)$ respectively: 

    We will show that the local action of $\Lambda_L = \pr_{X_L}(\Lambda)$ on $X_L$ is $\gen{\{\zeta_z\}_{z\in Z}}$ (the arguments for the other projection are similar).
    Note that since  $\Lambda$ acts transitively on the vertices of $X_Z\times X_L$, $\Lambda_L$ acts transitively on the vertices of $X_L$. Also, by assumption $\gen{\{\zeta_z\}_{z\in Z}}\le \Aut(L)$. We are left to show that $\bfD_v((\Lambda_L)_v) = \gen{\{\zeta_z\}_{z\in Z}}$  for some $v\in X_L$ and  $\Lambda_L\le U\big( \gen{\{\zeta_z\}_{z\in Z}}  \big)$. 
    
    For the first part,  consider  $X_Z\times\{v\}\subseteq X_Z\times X_L$, where $v$ is the vertex in $X_L$ corresponding to the trivial element in $W_L$. This is the Cayley graph associated to $W_Z = \gen{\{z\}_{z\in Z}}$. Therefore, the set-wise stabilizer of $X_Z\times \{v\}$ in $\Lambda$ is $W_Z = \gen{\{z\}_{z\in Z}}$, and $(\Lambda_L)_v = \pr_{X_L}(W_Z)$.  
    For every $z\in Z$, the projection $\pr_{X_L}(z)$ fixes the vertex $v$ in $X_L$, and $\bfD_v(\pr_{X_L}(z))=\zeta_z$.
    Therefore, 
    \begin{align*}
        \bfD_v((\Lambda_L)_v) &= \bfD_v(\pr_{X_L}(W_Z)) \\
        &= \bfD_v(\pr_{X_L} (\gen{\{z\}_{z\in Z}}))\\
        &=\bfD_v(\gen{\{\pr_{X_L} (z)\}_{z\in Z}}) \\
        &= \gen{\{\bfD_v(\pr_{X_L}(z))\}_{z\in Z}}\\
        &= \gen {\{\zeta_z\}_{z\in Z} }
    \end{align*}
    
    For the second part, consider $\{w\}\times X_L\subseteq X_Z\times X_L$ where $w$ is the vertex in $X_Z$ corresponding to the trivial element in $W_Z$. 
    This is the Cayley graph $X_L$ of $W_L = \gen{\{D\}_{D\in V(L)}}$.
    Thus, the action of $\pr_{X_L}(W_L)$ on $X_L$ is the usual action of $W_L$ on its Cayley graph. It satisfies for every $g\in \pr_{X_L}(W_L)$ that $\bfD_x(g)=\id$ for every vertex $x$ in $X_L$.
    
    Let $g\in\Lambda_L$, and $x$ is a vertex in  $X_L$. Then, we can write $g$ as the product $g=g_1hg_2$ where $g_1,g_2\in \pr_{X_L}(W_L)$, $g_2(x)=v$, and $h\in (\Lambda_L)_v$.
    Then, $$\bfD_xg = \bfD_v \circ g_1\bfD_vh \circ \bfD_xg_2 = \bfD_vh\in \gen{\{\zeta_z\}_{z\in Z}}.$$
    Since this holds for all $g\in\Lambda_L$, and for all $x$ in  $X_L$, we get $$\Lambda_L \le  U(\gen{\{\zeta_z\}_{z\in Z}}).$$

\end{proof}

\subsection{Construction of a lattice containing a BMW}
\label{lattice containing a BMW}
Recall that our goal, \Cref{building the lattice}, is to embed an involutive BMW group $\Gamma$ in a lattice $\Lambda\le \Aut(T_c) \times \Aut(X_L)$ with alternating local actions. The lattice $\Lambda$ will be constructed using \Cref{the lattice Lambda} by specifying some interlacing pair $\calD$. In this subsection we describe how to construct $\calD$ from an involutive BMW group of degree $(m,n)$ using an ``$n$-scaffolding'' $\calE$.

\subsubsection{The scaffolding}

\begin{definition}
For $n\in \bbN$, an $n$-\emph{scaffolding} is the data $\calE$ consisting of:
\begin{enumerate}[label = (E\arabic*)]
    \item \label{data E number} two natural numbers $k,d\in \bbN$ 
    so that $\max\{n+1,6\}\le d$; 
    \item \label{data E neighbors A'}a vertex $A'\in V(L)$, and a set $\calA'$ of $n+1$ (out of the $d$) neighbours of $A'$, which we denote $\calA'=\{A_1,\dots,A_n,B\}$;
    \item \label{data E upsilons}even involutions $\upsilon_1,\dots,\upsilon_k$ which generate $\Alt_{2d-1}$, with  ${\supp(\upsilon_1)\subseteq B}$;
    \item \label{data E sets C} a set $\calC=\{C_{1},\ldots,C_{k}\}\subseteq V(L)\setminus\calA'$ such that $\calA'\cup\calC$ is an independent set of vertices in the graph $L$, and $$\bigcup_{j=i}^{i+3}\supp(\upsilon_{j\pmod k})\subseteq C_{i}$$ for every $i\in[k]$.\footnote{We use the notation $\pmod k$ for the residue in $[k]=\{1,2,\ldots,k\}$.}
\end{enumerate}
\end{definition}

\begin{lemma}\label{existence of scaffoldings}
For every $n$ there exists an $n$-scaffolding.
\end{lemma}
\begin{proof}
    We write down an explicit $n$-scaffolding $\calE$:
\begin{itemize}
    \item $d=\max\{n+1,6\}$ and $k=2d-1$.
    \item $A'=\{5,6,7,9,11,13,15,\ldots,2d-1 \}= \{2i-1 | 3\leq i\leq d\} \cup\{6\}$.
    \item For $1\leq i\leq 4$, $A_i=[2d-1]\setminus (A'\cup  \{i\})$.

    For $5\leq i\leq n$, $A_i=[2d-1]\setminus(A'\cup  \{2i\})$.
    \item $B=\{1,2,3,4\}\cup\{10,12,14,16,\ldots,2d-2 \}= \{2i-2 | 6\leq i\leq d\} \cup\{1,2,3,4\}$.
    \item $\upsilon_i=(i,i+1)(i+2,i+3)\mod(2d-1)$.
    \item Define $C_i'=\{i,i+1,\ldots, i+d-3 \}_{\mod 2d-1} =\cup_{j=i}^{i+d-3}\{j\}$  ($\mod 2d-1$). 
    
    So $|C_i'|=d-2$, and finally,
    \item $C_i=C_i' \cup \begin{cases}
			\{1\}, & \text{if $1\notin C_i'$ }\\
            \{d\}, & \text{otherwise}.
		 \end{cases}$
\end{itemize} 
\end{proof}

\subsubsection{Presentation and local actions of involutive BMW groups.}
Every involutive BMW group $\Gamma$ of degree $(m,n)$ has a presentation of the following form:
\begin{equation}\label{BMW presentation}
\Gamma=\left< \calX=\{x_1,\ldots x_m\}\cup \calA=\{a_1,\dots,a_n\}\, |\,  \calR \right>
\end{equation}
where the set of relations $\calR$ consists of:
\begin{itemize}
    \item for all $1\le i\le m$, the relation $x_i^2\in \calR$, 
    \item for all $1\le j\le n$, the relation $a_j^2\in \calR$, and
    \item for all $1\le i\le m,1\le j\le n$, there is a unique relation $x_ia_jx_{i'}a_{j'}\in \calR$ up to cyclic permutation and inverses.
\end{itemize}  

\begin{remark}\label{rem: complete involutive BMW}
 In the relations of type $x_ia_jx_{i'}a_{j'}$, $i,i'$ are not necessarily distinct, and so are $j,j'$.
\end{remark}


The local action of the projection of $\Gamma$ to $T_n$ is the group 
$\gen{\{\xi_i\}_{i=1}^m}$ where $\xi_i$ is the permutation 
defined by the relations in $\calR$ 
in the following way:
For every $1\le j\le n$,  $\xi_i(a_j)=a_{j'}$ and $\xi_i(a_{j'})=a_{j}$, where either  $x_ia_jx_{i'}a_{j'}\in \calR$ or $x_ia_{j'}x_{i'}a_{j}\in \calR$, for some $1\le i'\le  m$.

Similarly, the local action of the projection of $\Gamma$ to $T_m$ is the group  $\gen{\{\alpha_j\}_{j=1}^n}$, where for every $1\le j\le n$ and for every $1\le i\le m$,  
$\alpha_j(x_i)=x_{i'}$ and $\alpha_j(x_{i'})=x_i$ if either  $x_ia_jx_{i'}a_{j'}\in \calR$ or $x_{i'}a_{j}x_{i}a_{j'}\in \calR$, for some $1\le j'\le  n$.

In words, when thinking of the elements in a relation $x_ia_jx_{i'}a_{j'}\in \calR$ as the vertices of a 4-cycle, then the permutation associated with a vertex switches between its two neighbours.

Note that by the properties of $\calR$, these permutations are well defined.





\subsubsection{Constructing an interlacing pair from an involutive BMW using a scaffolding}
\label{BMW+scaffolding defines interlacing pair}

Given an involutive BMW group $\Gamma$ of degree $(m,n)$ with the presentation \ref{BMW presentation} whose local actions are $\Alt_m$ and $\Alt_n$ respectively, and an $n$-scaffolding $\calE$, we will show how one can construct an interlacing pair $\calD=(\{\zeta_z\}_{z\in Z},\{\delta_D\}_{D\in V(L)})$ such that $\Gamma$ embeds in the group $\Lambda$ defined in \Cref{the lattice Lambda}.

First we describe the index sets $Z$ and $L=O_d$: As usual, $L=O_d$ (where $d$ is part of $\calE$, see \ref{data E number}). For $Z$ we take the disjoint union of $\calX = \{x_1,\dots,x_m\}$ and a set $\calY=\{y_1,\dots,y_k\}$. 

The involutions $\zeta_z$ and $\delta_D$ are defined as follows:
\begin{equation}\label{definition of zeta and delta}
    \zeta_z = \begin{cases} \xi_i & \text{if }z = x_i \in \calX \\ \upsilon_i & \text{if }z = y_i \in \calY \end{cases}
    \hspace{2cm}\delta_D = \begin{cases} \alpha_i & \text{if }D = A_i \in \calA' \\ \beta & \text{if }D = B \in \calA' \\ \gamma_i & \text{if }D=C_i \in \calC \\ \id & \text{otherwise}\end{cases}
\end{equation}
where the involutions $\upsilon_i$ are part of the $n$-scaffolding $\calE$ and the involutions $\xi_i, \alpha_i, \beta, \gamma_{i}$ are defined below:


    \textbf{Defining $\xi_i$.}  

For every $1\le i\le m$ we define $\xi_{i}\in\Alt_{2d-1}$ as follows:
The elements $A_1,\dots,A_n$ are neighbours of $A'$. Therefore, there is an embedding $\bar \cdot : [n] \mapsto [2d-1]$ defined by $\bar{j}$ is the unique element of $[2d-1] - (A_i \cup A')$.\footnote{In fact, Odd graphs were named so because each edge in the graph has an "odd one out", an element that does not participate in the two sets connected by the edge.}
The involution $\xi_i$ is defined as follows:
$$ \xi_i (k) = \begin{cases} \bar{j'} & \text{if }k=\bar{j}\text{ and }x_{i}\cdot a_{j}\cdot x_{i'}\cdot a_{j'}\in \calR\text{ for some }i'\\
k & \text{otherwise}\end{cases}$$

In other words, for every $x_{i}\cdot a_{j}\cdot x_{i'}\cdot a_{j'}\in \calR$, the involution $\xi_i$ contains the transposition $(\bar{j}\; \bar{j'})$ in its cycle decomposition. In particular, $\xi_i$ is an involution.

Alternatively, let $\xi'_i\in \Sym_n$ be the local action of $x_i\in \Gamma \le \Aut(T_m)\times \Aut(T_n)$ on the labels of the tree $T_n$ (i.e., on the set $[n]$). Then $\xi_i$ is the image of $\xi'_i$ under the embedding $\Sym_n \inj \Sym_{2d-1}$  induced by the embedding $[n]\inj [2d-1]$ given by $j\mapsto \bar{j}$.
Using this point of view, it follows that if the local actions of $\Gamma$ are alternating then 
\begin{equation}\label{xis are in alt}
\xi_1,\ldots,\xi_m\in \Alt_{2d-1}.
\end{equation}





\textbf{Defining $\alpha_i$.}  

For every $1\le j\le n$ we define $\alpha_{j}\in\Alt(Z)$ by
$$\alpha_j(z) = \begin{cases} 
x_{i'} & \text{if }z= x_i\text{ and }x_i a_j x_{i'} a_{j'}\in \calR\text{ for some }j'\\
z & \text{otherwise}
\end{cases}$$
Alternatively, it is the product of the following transpositions
$$ \alpha_{j}=\prod_{x_{i}\cdot a_{j}\cdot x_{i'}\cdot a_{j'}\in \calR} (x_{i}\;x_{i'}). $$
In particular, $\alpha_i$ is an involution.

This is the local action of $a_j\in \Gamma\le \Aut(T_m) \times \Aut(T_n)$ on the labels $\calX=\{x_1,\dots,x_n\}$ of the tree $\Aut(T_m)$, extended by the identity on $\calY$ to $Z=\calX \cup \calY$.

\textbf{Defining $\beta$.} 

We define $\beta$ to be the product of the two transpositions  $\beta=(x_1 \; y_1)(x_2 \; x_3)$. 

\textbf{Defining $\gamma_i$.} 

 For every $1\leq i\leq k$ we define $\gamma_{i}=(y_i, y_{i+1})(y_{i+2},y_{i+3})$, where the indices $i+j$ are taken $\pmod k$. 
 





\begin{lemma}\label{from involutive to interlacing}
    Given an involutive BMW group $\Gamma$ of degree $(m,n)$ and an $n$-scaffolding $\calE$, the pair $\calD = (\{\zeta_z\}_{z\in Z}, \{\delta_D\}_{D\in V(L)})$ defined above is interlacing.
\end{lemma}

\begin{proof} 
    By their definition, $\zeta_z, \delta_D$ are involutions in $\Alt(\ell),\Alt(Z)$ respectively, so \ref{zetas are involutions} and \ref{deltas are involutions} hold for $\calD$.
    
        \medskip
    
    \ref{interlacing deltaD=1 or deltaD'=1}: consider an edge $\{D,D'\} \in E(L)$. By \ref{data E sets C}, $\calA' \cup \calC$ is assumed to be an independent set in $L$. Thus, at least one of $D,D'$ is not in $\calA'\cup \calC$.  Therefore, by definition \eqref{definition of zeta and delta}, at least one of $\delta_D,\delta_{D'}$ is the identity.

    \medskip
    
    \ref{interlacing condition on zetas}: 
    Consider $z,z',D$ such that $\delta_D(z)=z'$, we want to show that $\zeta_z(D)=\zeta_{z'}(D)$ and $\zeta_z|_{[\ell]-D} = \zeta_{z'}|_{[\ell]-D}$. 
    
    If $z=z'$ there is nothing to prove. So if suffices to consider $z\in \supp (\delta_D)$. In particular, this proves the case $D\notin \calA' \cup \calC$ as in this case $\delta_D=\id$.

    If $D=A_{j}\in\calA'$, then by \eqref{definition of zeta and delta}, $\delta_D = \delta_{A_j}=\alpha_j$. 
    By the definition of $\alpha_j$ we have $\supp(\alpha_j)\subseteq \{x_1,\dots,x_m\}$.
    Let $z=x_i\in \supp(\alpha_j)$ and $z'= \alpha_j(x_i) = x_{i'} \in \supp(\alpha_j)$. 
    By \eqref{definition of zeta and delta}, $\zeta_z = \xi_i$ and $\zeta_{z'} = \xi_{i'}$.
    Since $\alpha_j(x_i)= x_{i'}$, By the definition of $\alpha_j$, this happens only if $x_{i}\cdot a_{j}\cdot x_{i'}\cdot a_{j'}\in \calR$ for some $j'$. 
    Now, by the definition of $\xi_i$ and $\xi_{i'}$, the relation $x_{i}\cdot a_{j}\cdot x_{i'}\cdot a_{j'}\in \calR$ implies that 
    $\xi_i(\bar{j}) = \bar{j}' = \xi_{i'}(\bar{j})$.
    Therefore,  $\delta_z(D)=\xi_{i}(A_{j})=A_{j'}=\xi_{i'}(A_{j}) = \delta_{z'}(D).$ Also, for $l\in [\ell]-A_j$, either $l\in A'$ in which case $\xi_i(l) = l = \xi_{i'}(l)$, or 
    $l = \bar{j}$ in which case $\xi_{i}(\bar{j})= \bar{j'} = \xi_{i'}(\bar{j}).$
    This shows that $\zeta_z|_{[\ell]-A_j} = \zeta_{z'}|_{[\ell] - A_j}$.

    If $D = B \in \calA'$ then $\delta_D= \delta_B = \beta$. 
    Let $z\in \supp(\delta_D) = \supp(\beta) = \{x_1,y_1,x_2,x_3\},$ and consider $\supp(\zeta_z)$. 
    If $z=y_1$ then by assumption \ref{data E upsilons}, $\supp(\zeta_{y_1}) = \supp(\upsilon_1) \subseteq B.$ Otherwise, $z\in \{x_1,x_2,x_3\}$ and by construction $\supp (\zeta_{x_i}) = \supp(\xi_i)\subseteq \{\bar{j} \;|\; j\in [m]\} \subseteq B.$
    In either case, $\supp (\zeta_z) \subseteq B$, and similarly $\supp(\zeta_{z'}) \subseteq B$. So $\zeta_z(B)=\zeta_z'(B)=B$ and $\zeta_z|_{[\ell]-B} = \zeta_{z'}|_{[\ell]-B} = \id_{[\ell]-B}$.


    If $D = C_i\in \calC$, then $\delta_D = \delta_{C_i}= \gamma_i$. 
    Let $z\in \supp(\delta_D) = \supp(\gamma_i) = \{y_i,y_{i+1},y_{i+2},y_{i+3}\},$ and consider $\supp(\zeta_z)$. For $i\le j \le i+3$, by the definition \eqref{definition of zeta and delta}, $\zeta_{y_j} = \upsilon_j$ and so, by assumption \ref{data E sets C}, $\supp(\zeta_{y_j}) = \supp(\upsilon_j) \subseteq C_i$. The desired result follows as before.


        \bigskip 
    
    \ref{interlacing condition on deltas}: Let $z,D,D'$ be such that $\zeta_z(D)=D'$. We want to show that 
    \begin{equation}\label{equality for deltas}
    \delta_D(z) = \delta_{D'}(z)
    \end{equation}
    
    
    If $z=x_i \in \calX$, then by \eqref{definition of zeta and delta} $\zeta_z = \xi_i$. 
    If $D=A_j \in \calA'$ for some $j$, denote $\bar{j'}=\xi_i(\bar{j})$, $\zeta_z(D)=\xi_i(A_j)=A_{j'}$. 
    By the definition of $\xi_i$, $x_i a_j x_{i'} a_{j'}\in\calR$. Hence, by the definition of $\alpha_j$ we also have $\delta_{D}(z) = \alpha_j(x_i) = x_{i'} = \alpha_{j'}(x_i) = \delta_{D'}(z)$.
    If $D=B$ then $\supp(\xi_i) \subseteq \{\bar{j} | j\in [m]\}\subseteq B$ so $\zeta_z(D)=D'$ and \eqref{equality for deltas} holds trivially.
    If $D\notin \calA'$ then by \eqref{definition of zeta and delta} $\delta_D = \gamma_i$ or $\delta_D=\id$ and in either case $\delta_D(z) =\delta_D(x_i)=z$ and similarly  $\delta_{D'}(z)=z$ so \eqref{equality for deltas} holds.

    If $z=y_i \in \calY$, then by \eqref{definition of zeta and delta} $\zeta_z = \upsilon_i$. 
    If $D=C_j \in \calC$ for $i-3\le j\le i$, by \ref{data E sets C}, $\supp(\upsilon_i) \subseteq C_j$ and so $D'=\zeta_z(D) = \upsilon_i(C_j) = C_j$ and \eqref{equality for deltas} trivially holds.
    If $D=B \in \calA'$ and $z=y_1$ then by \ref{data E sets C}, $\supp(\upsilon_1) \subseteq B$ and so $D' = \zeta_z(D) = \upsilon_1(B)=B$ and \eqref{equality for deltas} trivially holds.
    In all other cases \eqref{definition of zeta and delta} implies $\delta_D(z) = z = \delta_{D'}(z)$, and again \eqref{equality for deltas} holds.
    
    

    

    

    
\end{proof}

\subsection{Proof of \Cref{building the lattice}}
\label{proof of building the lattice}

\begin{proof}[Proof of \Cref{building the lattice}]
Let $\Gamma$ be an involutive BMW of degree $(m,n)$, and let $\calE$ be an $n$-scaffolding (\Cref{existence of scaffoldings}).
By \Cref{from involutive to interlacing} the pair $\calD$ defined in \S\ref{BMW+scaffolding defines interlacing pair} is interlacing. 
By \Cref{the lattice Lambda} $\calD$ gives rise to a lattice $\Lambda \le \Aut(T_{c})\times \Aut(X_{O_d})$ where $c=m+k$ and $k,d$ are given by the scaffolding data $\calE$.
 In the next two claims we will show that the local actions are alternating and $\Gamma$ embeds in $\Lambda$.

\begin{claim}\label{the local actions of Lambda}
    The local actions of $\Lambda\le \Aut(X_Z) \times \Aut(X_L)$ are $\Alt(Z)$ and $\Alt_\ell$.
\end{claim}

\begin{proof}
     By \Cref{the lattice Lambda}, it suffices to show that $\gen{\{\delta_D\}_{D\in V(L)}} = \Alt(Z)$ and $\gen{\{\zeta_z\}_{z\in Z}} = \Alt_\ell$.
     
     First we note that $\delta_D \in \Alt(Z)$ for all $D\in V(L)$, proving one inclusion. For the other inclusion,
     we assumed that the local action of $\Gamma$ is $\Alt(\calX)$, so we have $\gen{\alpha_1,\ldots,\alpha_n}=\Alt_{\calX}$.
     By the definition of $\gamma_i$ we have  
     $\gen{\gamma_1,\ldots,\gamma_k}=\Alt_{\calY} $.
     Since $\Alt(\calX)\times\Alt(\calY)$ is a maximal subgroup of $\Alt(\calX\cup\calY )=\Alt(Z)$ when $|\calX|\ne |\calY|$, and $\beta\in\Alt_{Z} \setminus (\Alt_{\calX}\times\Alt_{\calY})$, it follows that $$\gen{\{\delta_D\}_{D\in V(L)}} = \gen{\alpha_1,\ldots,\alpha_n,\beta,\gamma_1,\ldots,\gamma_k}=\Alt_{ \calX\cup\calY }=\Alt(Z).$$

     For the other local action, by \ref{data E upsilons} we have $$\gen{\{\zeta_z\}_{z\in Z}} = \gen{\xi_1,\dots,\xi_m,\upsilon_1,\dots,\upsilon_k} = \Alt_\ell.$$
\end{proof}


\begin{claim}\label{Gamma embeds in Lambda}
    The group $\Gamma$ embeds as a subgroup of the stabilizer of a hyperplane in $X_Z\times X_L$ in $\Lambda$.
\end{claim}

\begin{proof}
    Define the map 
     $\iota:\Gamma\to\Lambda$ such that $\iota(x_i)=x_i$ and $\iota(a_j)=A_j$. By the construction of $\xi_i$'s and $\alpha_j$'s we have that if $x_{i_1}\cdot a_{j_1}\cdot x_{i_2}\cdot a_{j_2}\in \calR$, then $x_{i_1}\cdot A_{j_1}\cdot x_{i_2}\cdot A_{j_2}\in \calQ$, where $\calQ$ is the set of relations in $\Lambda$.

    The image of $\iota$ is contained in the stabilizer of the hyperplane $\calH' = X_Z \times \calH$ where $\calH$ is the hyperplane labelled $A'\in V(L)$ in $X_L$.
     
     To show that it is an embedding, we note that the map $\iota$ induces a map from the Cayley graph of $\Gamma$ to that of  $\Lambda$, thus induces
     a map between $f:T_m\times T_n \to \calH'$ that is 
     equivariant in the sense that $f(\gamma x)=\iota(\gamma)f(x)$ for all $\gamma\in \Gamma$ and $x\in T_m\times T_n$.
     
     $f$ is a local isometry: let $v$ be a vertex in the Cayley graph of $\Gamma$. By the definition of $\iota$, if  $e_1\ne e_2$ are edges incident to $v$,  then $f(e_1)\ne f(e_2)$. If there is a square $P$ with $e_1,e_2\in\partial P$, it corresponds to a square in $\calR$, therefore is sent to a square in $\calQ$ incident to $v$.
     
     So $f$ is local isometry between CAT(0) spaces, so it is an embedding.  Therefore if the action of $\gamma$ on its Cayley graph is non-trivial then so is $\iota(\gamma)$, which implies that $\iota$ is an embedding.
     \end{proof}
This finishes the proof of \Cref{building the lattice}.
\end{proof}



\section{Proof of \Cref{simple Lattices In Products}.}\label{proof section}

We recall \Cref{simple Lattices In Products}: 
\simpleLatticesInProducts*

\begin{proof}[Proof of \Cref{simple Lattices In Products}] 
For infinitely many $m,n\ge 6$ there exists an involutive BMW group $\Gamma$ of degree $(m,n)$  which is non-residually-finite with alternating local actions. 
See \cite[Theorem 5.5]{radu2020new} or  \cite[Theorem A]{lazarovich2022counting} (see also \cite{burger2000groups} for non-involutive BMW groups).
For each such $\Gamma$, by  \Cref{building the lattice} there exist $c,d\ge 6$ and irreducible lattices $\Lambda' \le \Aut(T_{c}) \times \Aut(X_{O_d})$ such that 
\begin{enumerate}
    \item $\pr_{T_{c}}(\Lambda')$ and $\pr_{X_{O_d}}(\Lambda')$ are vertex transitive on $T_c$ and $X_{O_d}$ respectively.
    \item $\pr_{T_{c}}(\Lambda')$ and $\pr_{X_{O_d}}(\Lambda')$ are non discrete in $\Aut(T_c)$ and $\Aut(X_{O_d})$ respectively. 
    \item The local action of $\Lambda'$ on both $\Aut(T_c)$ and $X_{O_d}$ are alternating.
\end{enumerate}

Then, by \cite[Lemma 3.5.3]{burger2000groups}  for $T_c$, and \Cref{local to global} for $X_{O_d}$ we have that 
$\overline{\pr_{T_{c}}(\Lambda')}=U_{T_c}$ and $\overline{\pr_{X_{O_d}}(\Lambda')}=U_{O_d}$.
By \cite[Proposition 3.2.1]{burger2000groups}, 
the subgroup $U^+_{T_c}\le U_{T_c}$ is an index 2 closed simple subgroup. By \Cref{simplicity of U}, 
$U^+_{O_d}\le U_{O_d}$ is an index 2 closed simple subgroup.

By Bader-Shalom \cite{bader2006factor}, 
$\Lambda'$, and any of its finite index subgroups, is just-infinite, i.e., has no non-trivial infinite-index normal subgroups. 

In addition,  $\Lambda'$ is non-residually finite  because it contains the non-residually finite group $\Gamma$, by \Cref{building the lattice}. Therefore 
its finite residue $\Lambda=\bigcap \{H\le \Lambda' \;\mid\; |\Lambda':H|<\infty\}$  is a non-trivial normal subgroup. 

Since $\Lambda'$ is just-infinite,  $\Lambda$ has finite index in $\Lambda'$, and therefore is itself just infinite.

So $\Lambda$ is just-infinite, has no non-trivial infinite-index normal subgroups, but by its definition (and the fact that it is finite-index in $\Lambda'$), $\Lambda$ has no proper finite-index subgroups, in particular no finite-index normal,  and so it is simple. 

By the construction in $\Cref{construction}$ we see that $\Lambda'$ is finitely presented, and since $|\Lambda':\Lambda|<\infty$, $\Lambda$ is also finitely presented. 
\end{proof}


\section{Discussion}

\paragraph{Torsion and cohomological dimension.}
All known examples of finitely presented simple group have cohomological dimension $2$ or $\infty$.

\begin{question}\label{question about cd}
    Are there finitely presented simple group of cohomological dimension other than $2$ or $\infty$?
\end{question}

Since torsion free uniform lattices in the product $T_c\times X_{O_d}$ have cohomological dimension 3, finding a torsion free lattice as in \Cref{simple Lattices In Products} will give a positive answer to the above question. 

The group $\Lambda'$ constructed in the proof of \Cref{simple Lattices In Products} clearly has torsion as it is generated by involutions. 
Even the index 4 subgroup $\Lambda'^+=\Lambda' \cap (\Aut^+(T_c)\times \Aut^+(X_{O_d}))$ of $\Lambda'$ has torsion, namely, when $[D,D']=1$ the involution $DD'$ is contained in $\Lambda'$.
However, it is unclear whether the finite index simple subgroup $\Lambda$ contained in $\Lambda'^+$ has torsion. 


\begin{question}
    Is there a construction of a virtually simple lattice $\Lambda'$ in $\Aut(T_c)\times \Aut(X_{O_d})$ such that $\Lambda'^+$ is torsion free?
\end{question}

As explained above, a positive answer to this question will imply a positive answer to \Cref{question about cd}.
 



\paragraph{Lattices in other complexes.}

\begin{question}
    Are
    there  finitely-presented, simple, uniform lattice in $\Aut( X_{O_d} )\times \Aut(X_{O_{d'}})$?
\end{question}

Note that $X_{O_d}\times X_{O_{d'}}$ is four-dimensional.
Since \Cref{local to global} applies to both factors, it suffices to construct a non-residually finite lattice in the product $X_{O_d}\times X_{O_{d'}}$ whose projections have the desired local actions.

The case of using the Odd graph can also be generalized:

\begin{definition}
 The Kneser complex $K(n,k)$ is the simplicial complex whose vertices are all subsets of size $k$ out of the set $[n]=\{1,2,\ldots,n\}$, and simplices are pairwise disjoint subsets.  
\end{definition}

Note that $K(2k+1,k)=O_{k+1}$. 
It was shown in \cite{lazarovich2018regular} that $K(nd+1,n)$ is a superstar-transitive $(d-1)$-dimensional flag simplicial complex, and therefore the Davis complex $X_{K(nd+1,n)}$ is the unique CAT(0) cube complex with link $K(nd+1,n)$ at each vertex.

\begin{question}
    For $d,d',n,n'\in \bbN$, 
    are
    there finitely-presented, simple, uniform lattices in $\Aut( X_{K(nd+1,n)} )\times \Aut(X_{K(n'd'+1,n')})$?
\end{question}



A different direction could be to find such lattices in automorphism group of a regular tree and a Bourdon building. 
\begin{definition}
    For $5\leq p\in \bbN, q\in \bbN$, a \textit{Bourdon building $I_{p,q}$}, is a building such that the apartments are tessellations of the hyperbolic plane by right angled $p$-gons, and the link at any vertex is $K_{q,q}$. 
\end{definition}

For even $p=2m$, the Bourdon building is isomorphic to the Davis complex of the Coxeter group $$W=\gen{x_1,\dots,x_m,y_1,\dots,y_m\;|\;x_i^2, \; y_j^2,\; (x_iy_j)^m, \forall 1\le i,j\le m}.$$
This gives an edge-colouring of $I_{p,q}$ with the colours $x_1,\dots,x_m,y_1,\dots,y_m$, and one can use it to define a universal group whose local action is $\Alt(\{x_1,\dots,x_m\})\times\Alt(\{y_1,\dots,y_m\})$.

\paragraph{Commensurators in $\Aut(X_{O_d})$.}
The commensurator of $\Gamma$ in $G$ is the group $$\Comm_{G}(\Gamma) = \{g\in G\;|\; |\Gamma:\Gamma \cap g\Gamma g\ii|<\infty\; \text{ and }|g\Gamma g\ii:\Gamma \cap g\Gamma g\ii| <\infty\}.$$
Leighton \cite{leighton1982finite} and Bass-Kulkarni \cite{bass1990uniform} proved that any two uniform lattices in $\Aut(T)$ are commensurable up to conjugacy. 
This shows that the group $C=\Comm_{\Aut(T)}(\Gamma)$ for a uniform lattice $\Gamma$ in independent of the specific uniform lattice $\Gamma$. 
Lubotzky, Mozes, and Zimmer \cite{lubotzky1994superrigidity} asked if $C^+$ is simple. 
This was recently disproved, see \cite{barnea2025commensurators, leboudec2025commensurator}.

\begin{question}
Is the commensurator of a uniform lattice in $\Aut(X_{O_d})$ simple?
\end{question}

In \cite{woodhouse2023leighton}, Woodhouse proved Leighton's Theorem for the Davis complex $X_L$ for $L=O_d$ (in fact, for the more general $L = K(dn+1,n)$, cf. also \cite{haglund2006commensurability}).
In particular the commensurator $\Comm_{\Aut(X_L)}(W)$ of a uniform lattices $W \le \Aut(X_L)$ is independent of the specific uniform lattices $W$.
The virtually simple group $\Lambda'$ constructed in \Cref{simple Lattices In Products} commensurates the Coxeter group $W_L \le \Lambda'$. Therefore, the projection $\pr_{X_L}(\Lambda')$ commensurates the lattice $W = \pr_{X_L}(W_L )\le \Aut(X_L)$. Since $\Lambda'$ is virtually simple $\Lambda' \simeq\pr_{X_L}(\Lambda')$ is in the commensurator $\Comm_{\Aut(X_L)}(W)$.

Caprace in \cite[Appendix A]{radu2020new} used simple lattices in products of trees to prove that $C^+$ is almost simple -- that is, the intersection of all non-trivial normal subgroups is a non-trivial, non-abelian simple group.
\begin{corollary}
    $G=\Comm_{\Aut(X_L)}(W)$ is almost simple. 
\end{corollary}

\begin{proof}
By Lemma A.4 of \cite{radu2020new}, to show that $G$ is almost simple it suffices to show that $G$ contains a simple group $B$ which contains a finite index subgroup of $W$, and that there is no non-trivial element $z\in G$ which centralizes a finite index subgroup of $W$.

Let $B = \pr_{X_{O_d}}(\Lambda)$. Then $B$ is simple. 
Since $|\Lambda':\Lambda|<\infty$ and $W\le \pr_{X_{O_d}}(\Lambda')$ we have  $|W:W\cap B|<\infty$.
Thus, the simple group $B$ contains a finite index subgroup of $W$.

Let $z\in G$ be an element that centralizes a finite index subgroup of $W$. 
Since $L$ has no induced 4-cycles, $W$ is hyperbolic. 
The action of $\Aut(X_L)$ (and so in particular of $G$) on  $\partial W$ is faithful. 
Since $z$ centralizes $W$, it acts trivially on the Gromov boundary $\partial W_L$.
Therefore, $z=1$.
\end{proof}

\bibliographystyle{plain}
\bibliography{biblio}

 \end{document}